\documentclass[12pt,reqno]{amsart}
\usepackage{amssymb}
\usepackage{amsmath}
\usepackage{amsthm}
\usepackage[left=3cm,top=3cm,right=3cm,bottom=3cm]{geometry}
\usepackage{hyperref}
\usepackage[usenames]{color}
\usepackage{enumitem} 
\usepackage[normalem]{ulem} 

\usepackage{etoolbox}
\makeatletter
\patchcmd{\@maketitle}
  {\ifx\@empty\@dedicatory}
  {\ifx\@empty\@date \else {\vskip3ex \centering\footnotesize\@date\par\vskip1ex}\fi
   \ifx\@empty\@dedicatory}
  {}{}
\patchcmd{\@adminfootnotes}
  {\ifx\@empty\@date\else \@footnotetext{\@setdate}\fi}
  {}{}{}
\makeatother

\begin{document}
\newtheorem{theorem}{Theorem}[section]
\newtheorem{lemma}[theorem]{Lemma}
\newtheorem{definition}[theorem]{Definition}
\newtheorem{conjecture}[theorem]{Conjecture}
\newtheorem{proposition}[theorem]{Proposition}
\newtheorem{claim}[theorem]{Claim}
\newtheorem{corollary}[theorem]{Corollary}
\newtheorem{observation}[theorem]{Observation}
\newtheorem{problem}[theorem]{Open Problem}

\newtheorem*{c1}{\textbf{Chernoff Bound}}
\newtheorem*{c2}{\textbf{Markov Bound}}
\newtheorem*{ub}{\textbf{Union Bound}}

\theoremstyle{remark}
\newtheorem{rem}[theorem]{Remark}

\newcommand{\whp}{{\textit{w.h.p. }}}
\newcommand{\bin}{\textrm{Bin}}
\newcommand{\E}{\mathbb{E}}
\newcommand{\RT}{\textup{\textbf{RT}}}
\newcommand{\f}{\textup{\textbf{f}}}
\newcommand{\z}{\textrm{\textbf{z}}}
\newcommand{\ex}{\textrm{\textbf{ex}}}
\newcommand{\sqbs}[1]{\left[ #1 \right]}
\newcommand{\of}[1]{\left( #1 \right)}

\title[On the Ramsey-Tur\'an number]{On the Ramsey-Tur\'an number with small $s$-independence number}

\author{Patrick Bennett and Andrzej Dudek}
\address{Department of Mathematics, Western Michigan University, Kalamazoo, MI, USA}
\email{\tt \{patrick.bennett,\;andrzej.dudek\}@wmich.edu}
\thanks{The second author was sponsored by the National Security Agency under Grant Number H98230-15-1-0172. The United States Government is authorized to reproduce and distribute reprints notwithstanding any copyright notation hereon.}

%\keywords{TODO}
%\subjclass{TODO}

\begin{abstract}
Let $s$ be an integer, $f=f(n)$ a function, and $H$ a graph. Define the \emph{Ramsey-Tur\'an
number} $\RT_s(n,H, f)$ as the maximum number of edges in an $H$-free graph~$G$ of order $n$ with $\alpha_s(G) < f$, where $\alpha_s(G)$ is the maximum number of vertices in a $K_s$-free induced subgraph of $G$. The Ramsey-Tur\'an number attracted a considerable amount of attention and has been mainly studied for $f$ not too much smaller than~$n$. In this paper we consider $\RT_s(n,K_t, n^{\delta})$ for fixed $\delta<1$. We show that for an arbitrarily small $\varepsilon>0$ and $1/2<\delta< 1$, $\RT_s(n,K_{s+1}, n^{\delta}) = \Omega(n^{1+\delta-\varepsilon})$ for all sufficiently large $s$. This is nearly optimal, since a trivial upper bound yields $\RT_s(n,K_{s+1}, n^{\delta}) = O(n^{1+\delta})$. Furthermore, the range of $\delta$ is as large as possible.
We also consider more general cases and find bounds on $\RT_s(n,K_{s+r},n^{\delta})$ for fixed $r\ge2$. Finally, we discuss a phase transition of $\RT_s(n, K_{2s+1}, f)$ extending some recent result of Balogh, Hu and Simonovits.
\end{abstract}

\date{\today}

\maketitle

\section{Introduction}\label{sec:intro}

The \emph{$s$-independence number} of a graph $G$, denoted $\alpha_s(G)$, is the maximum number of vertices in a $K_s$-free induced subgraph of $G$ (so the standard independence number is the same as the $2$-independence number). For a given graph $H$, the \emph{Ramsey-Tur\'an number} $\RT_s(n, H, f)$ is the maximum number of edges in any $H$-free graph $G$ on $n$  vertices with $\alpha_s(G) < f$. If there does not exist any $H$-free graph $G$ on $n$ vertices with $\alpha_s(G) < f$ then we put $\RT_s(n,H,f)=0$.  Observe that the lower bound $k \le \RT_s(n,H,f)$ means that there exists an $H$-free graph~$G$ of order~$n$ with $\alpha_s(G) < f$ and at least $k$ edges. The upper bound $\RT_s(n,H,f) < \ell$ says that there is no $H$-free graph~$G$ of order $n$, $\alpha_s(G) < f$, and at least $\ell$ edges.

In general it is far from trivial to even determine the existence of any $H$-free graph~$G$ on $n$ vertices with $\alpha_s(G) <f$, let alone to maximize the number of edges in such graphs if they do exist. The \emph{Erd\H{o}s-Rogers number} $\f_{s, t}(n)$ is the minimum possible $s$-independence number taken over all $K_t$-free graphs $G$ of order~$n$. Note that if $f \le \f_{s,t}(n)$ then $\RT_s(n, K_t, f)=0$. Erd\H{o}s and Rogers \cite{ERog} were the first who studied $\f_{s,t}(n)$ for fixed $s$ and $t=s+1$ and $n$ going to infinity. They  proved that for every $s$ there is a positive $\varepsilon(s)$ such that $\f_{s, s+1}(n) \le n^{1-\varepsilon(s)}$, where $\lim_{s\to\infty}\varepsilon(s)=0$. This question was subsequently addressed by Bollob\'as and Hind~\cite{BH}, Krivelevich~\cite{KR2, KR}, Alon and Krivelevich \cite{AK}, Dudek and R\"odl \cite{DR}, Wolfovitz~\cite{Wo}, and most recently by Dudek and Mubayi~\cite{DM}, and Dudek, R\"odl and Retter~\cite{DRR}. Due to~\cite{DM} and~\cite{DRR} it is known that for any $s\ge 3$,
\begin{equation}\label{f:1}
\Omega\left(\sqrt{\frac{n \log n }{\log\log n}} \right) = \f_{s, s+1}(n) = O\left( (\log n)^{4s^2}\sqrt{n} \right).
\end{equation}
For $s=2$ very sharp bounds are known: 
\begin{equation}\label{f:1b}
(1/\sqrt{2} - o(1))\sqrt{n\log n} \le \f_{2,3}(n) \le (\sqrt{2} + o(1))\sqrt{n\log n},
\end{equation}
where the upper bound was obtained by Shearer~\cite{Shearer} and the lower bound, independently, by Bohman and Keevash~\cite{BK2}, and Pontiveros, Griffiths, and Morris~\cite{PGM}.
Furthermore, a result of Sudakov~\cite{SU} together with~\cite{DRR} imply that for any $\varepsilon > 0$ and an integer $r\ge 2$
\begin{equation}\label{f:r}
\Omega\big{(}n^{\frac{1}{2} - \varepsilon}\big{)} = \f_{s,s+r}(n) = O(\sqrt{n})
\end{equation}
for all sufficiently large~$s$.

The above bounds are quite recent and therefore the Ramsey-Tur\'an number\linebreak $\RT_s(n, H, f)$ was not extensively studied for small $f$. 
Previously, most researchers investigated the Ramsey-Tur\'an number for $f$ not too much smaller than $n$. Perhaps the first paper written about this problem was by S\'os~\cite{VS}. In particular there are many results on $\theta_r(H)$, where
\[
\theta_r(H) = \lim_{\varepsilon \rightarrow 0} \lim_{n \rightarrow \infty} \frac{1}{n^2}\RT(n, H, \varepsilon n).
\] 
It is perhaps surprising that $\theta_r(H)$ would ever be positive, but it is known for example that $\theta_2(K_4)= \frac{1}{8}$ (upper bound by Szemer\'edi in \cite{Sz}, lower bound by Bollob\'as and Erd\H{o}s in \cite{BE}). Several other exact values for $\theta_r(H)$ are known, and even more bounds are known where exact values are not, see, for example, the papers by Balogh and Lenz~{\cite{BL,BL2}}; Erd\H{o}s, Hajnal, Simonovits, S\'os, and Szemer\'edi~\cite{EHSSS}; 
Erd\H{o}s, Hajnal, S\'os, and Szemer\'edi~\cite{EHSS}, and Simonovits and S\'os~\cite{SS}.  Until recently only a very few results for $f \le n^{\delta}$ with $\delta<1$ were given. For example, Sudakov~\cite{SU} gave upper and lower bounds on $\RT_2(n, K_{4}, n^{\delta})$ (see also \cite{SU3} and~\cite{FLZ}). Balogh, Hu, and Simonovits~\cite{BHS} studied $\RT_2(n,K_t,f)$ for several pairs of $t$ and $f$.

In this paper we study $\RT_s(n, K_{s+r}, f)$ for $\f_{s,s+r}(n) +1 \le f\le n^{\delta}$ with $\delta<1$. In view of \eqref{f:1} and \eqref{f:r}, it is of interest to ask about $\RT_s(n, K_{s+r}, n^\delta)$ for $r\ge 1$ and $1/2 < \delta < 1$. Moreover, for $r\ge 2$ it also makes sense to ask for $\RT_s(n, K_{s+r}, Cn^{1/2})$ for a sufficiently large constant $C$ (which may depend on $s$). We will show (Theorem~\ref{thm:main1}, \ref{thm:main2}, and \ref{thm:ub_delta}) that for all $r\ge 1$, $\varepsilon>0$ and $1/2 < \delta < 1$, and all sufficiently large $s$,
\[
\Omega\of{n^{2 - (1-\delta)/r  - \varepsilon}} = \RT_s(n, K_{s+r}, n^\delta) <
\begin{cases}
n^{2-\frac{(1-\delta)^2}{r-\delta}}, & \text{ if $\frac{r-\delta}{1-\delta}$ is an integer},\\ 
n^{2-\frac{(1-\delta)^2}{r+1-2\delta}}, & \text{ otherwise}. 
\end{cases}
\]
In particular, this implies for $r=1$ and $1/2 < \delta < 1$ nearly optimal bounds,
\[
\Omega\of{n^{1+\delta  - \varepsilon}} = \RT_s(n, K_{s+1}, n^\delta) = O\of{n^{1+\delta}}.
\]
As a matter of fact the upper bound is trivial since as it was already observed in~\cite{EHSSS} if $G$ is $K_{s+1}$-free, then $\Delta(G) < \alpha_s(G)$.
We will also show that for specific values of $\delta$ one can remove $\varepsilon$ from the exponent (see Theorem~\ref{thm:quad} and Corollary~\ref{cor:quad}) and we conjecture that this should be true for all $1/2 < \delta < 1$.

Very recently Balogh, Hu, and Simonovits~\cite{BHS} introduced an interesting concept of a (Ramsey-Tur\'an) phase transition. Following this direction we will study behavior of $\RT_s(n, K_{2s+1}, f)$. In particular, for $1/2 < \delta < 1$ we show that $\RT_s(n, K_{2s}, n^\delta)$ is subquadratic, while $\RT_s(n, K_{2s+1}, n^\delta)$ is quadratic (and we actually find its value asymptotically exactly). But for $\delta = 1/2$, $\RT_s(n, K_{2s+1}, n^\delta)$ is subquadratic again, yielding a ``jump''  in the critical window.

Our proofs of the lower bounds are built upon the ideas in~\cite{DR,DRR,Wo} and are based on certain models of  random graphs that are constructed using some finite geometries. Roughly speaking, finite geometries provide us with a structure that allows us to bound the number of vertices that interact in certain ways, which helps us show that in the random graph we construct, we do not expect to see the forbidden subgraph (say if we are doing $\RT_s(n, K_{s+r}, f)$ then the forbidden subgraph is $K_{s+r}$) while we do expect to see many copies of $K_s$. The proofs use the probabilistic method, but the use of probability is relatively elementary. The proofs of the upper bounds use the dependent random choice technique (see, e.g., \cite{FS}).

The rest of the paper is structured as follows. In Section~\ref{sec:prelim} we state the probabilistic tools we will use. In Sections~\ref{sec:main1} and~\ref{sec:main2} we prove lower bounds on $\RT_s(n, K_{s+r}, n^{\delta})$ for $\delta$ close to $1/2$ (in Section~\ref{sec:main1}) and then for larger values $\delta <1$ (in Section~\ref{sec:main2}). In Section~\ref{sec:ub} we prove upper bounds on $\RT_s(n, K_{s+r}, n^{\delta})$. In Section~\ref{sec:r_eq_1} we further discuss the case $r=1$ and show that for a few values $\delta$ we can prove upper and lower bounds matching up to a constant factor. In Section~\ref{sec:big_r} we discuss a phase transition of $\RT_s(n, K_{2s+1}, f)$.

\section{Preliminaries}\label{sec:prelim}

This paper uses the probabilistic method in the most classical sense: if we define a random structure and show that with some positive probability the random structure has a certain property, then there must exist a structure with that property. The probabilistic aspect of this paper is elementary. We use only standard bounds on the probability of certain events, which we state here. 

We state basic forms of the Chernoff and Markov bounds (see, e.g., \cite{AS,JLR}).
\begin{c2}
If X is any nonnegative random variable and $\zeta > 0$, then
\[
\Pr\left(X \geq \zeta\cdot \E(X)\right) \leq \frac{1}{\zeta}.
\]
\end{c2}

Let $\bin(n,p)$ denotes the random variable with binomial distribution with number of trials~$n$ and probability of success~$p$.
\begin{c1}
If $X \sim \bin(n,p)$ and $0 < \varepsilon \leq \frac{3}{2}$, then
\[
\Pr \left( |X - \E(X)| \geq \varepsilon \cdot \E(X) \right) \leq 2 \exp \left\{ -\frac{\E(X) \varepsilon^2 }{3} \right\} .
\]
\end{c1}

We will also use the union bound.
\begin{ub}
If $E_i$ are events, then 
\[
\Pr \Big( \bigcup_{i=1}^k E_i \Big) \leq k \cdot \max \{\Pr(E_i): i \in [k]\}.
\]
\end{ub}

Finally, we say that an event $E_n$ occurs \emph{with high probability}, or \whp for brevity, if $\lim_{n\rightarrow\infty}\Pr(E_n)=1$.

All logarithms in this paper are natural (base $e$). Asymptotic notation can be viewed as either in the variable $n$ (the number of vertices in the graphs we are interested in) or $q$ (another parameter that will go to infinity along with $n$). When we say that a statement \emph{holds for $s$ sufficiently large}, we mean that there exists some $s_0$ (which may dependent on some other parameters) such that the statement holds for any $s \ge s_0$. For simplicity, in the asymptotic notion we do not round numbers that are supposed to be integers either up or down. This is justified since these rounding errors are negligible.

%%%%%%%%%%%%%%%%%%%%%%%%%%%%%%%%%%%%%%%%%%%%%%%%%%%%%%%%%%%%
%%%%%%%%%%%%%%%%%%%%  THE HYPERGRAPH H  %%%%%%%%%%%%%%%%%%%%
%%%%%%%%%%%%%%%%%%%%%%%%%%%%%%%%%%%%%%%%%%%%%%%%%%%%%%%%%%%%

\section{Lower bound on $\RT_s(n, K_{s+r}, n^\delta)$ for small values of $\delta$}\label{sec:main1}

In this section we will construct (randomly) a graph $G_1$ that gives our lower bound for $\RT_s(n, K_{s+r}, n^\delta)$ for relatively small $\delta$ (not much bigger than $1/2$). The construction uses several ideas from~\cite{DR,DRR,Wo}.

We start with the affine plane, which is a  hypergraph with certain desirable properties. For this range of $\delta$ we want a somewhat sparser hypergraph, so we randomly remove some edges from the affine plane to form a new hypergraph $\mathcal{H}_1$. We then construct $G_1$ by taking the vertices in each edge of $\mathcal{H}_1$ and putting a complete $s$-partite graph (with a random $s$-partition) on them together with a large independent set. Consequently, $G_1$ will have many copies of $K_s$. On the other hand, since each edge of $\mathcal{H}_1$ contains only copies of $K_s$ (and no larger complete subgraph), any possible copy of $K_{s+r}$ in $G_1$ must not be entirely contained in one edge of $\mathcal{H}_1$. We will exploit the properties  $\mathcal{H}_1$ and how its edges interact to show that this is unlikely, and therefore we do not expect to see any $K_{s+r}$ in $G_1$.

\subsection{The hypergraph $\mathcal{H}_1$} \label{sec:H1}

The {\em affine plane} of order $q$ is an incidence structure on a set of $q^2$ points and a set of $q^2+q$ lines such that: any two points lie on a unique line; every line contains~$q$ points; and every point lies on~$q+1$ lines. It is well known that affine planes exist for all prime power orders. (For more details see, e.g., \cite{CA}.)
Clearly, an incidence structure can be viewed as a hypergraph with points corresponding to vertices and lines corresponding to hyperedges; we will use this terminology interchangeably.

In the affine plane, call lines $L$ and $L'$ \emph{parallel} if $L \cap L' = \emptyset$. In the affine plane there exist $q+1$ sets of $q$ pairwise parallel lines. Let $\mathcal{H}=(V, \mathcal{L})$ be the hypergraph  obtained by removing a parallel class of $q$ lines from the affine plane or order~$q$. Thus, $\mathcal{H}$ is $q$-regular hypergraph of order~$q^2$.

The objective of this section is to establish the existence of a certain hypergraph $\mathcal{H}_1 = (V_1,\mathcal{L}_1) \subseteq \mathcal{H}$ by considering a random sub-hypergraph of $\mathcal{H}$. Preceding this, we introduce some terminology. 
Call $S \subseteq V_1$ \emph{complete} if every pair of points in $S$ is contained in some common line in~$\mathcal{L}_1$. 
We distinguish 2 types of complete \emph{dangerous} subsets $S\subseteq V$. \emph{Type~1} dangerous set consists of $|S|$ points in general position. \emph{Type~2} dangerous set consists of $|S|-r$ points that lie on some line $L\in\mathcal{L}_1$ and a set $R$ of $r$ many other points that do not belong to $L$.

\begin{lemma} \label{lem:1}
Let $q$ be a sufficiently large prime and $\log q \ll \lambda \le q$. 
Then, there exists a $q$-uniform hypergraph $\mathcal{H}_1=(V_1,\mathcal{L}_1)$ of order $q^2$ such that:
\begin{enumerate}[label=$({\textup{H}}_{1}\textup{\alph*})$]
\item\label{H1:a} Any two vertices are contained in at most one hyperedge;
\item\label{H1:b} For every $v \in V_1$, $\frac{\lambda}{2} \le \deg_{\mathcal{H}_1}(v) \leq \frac{3\lambda}{2}$; 
\item\label{H1:c} $|\mathcal{D}_a|  \leq  \frac{4\lambda^{a^2/2} }{ q^{(a^2-5a)/2}}$ and $| \mathcal{D}_b| \leq \frac{4\lambda^{br } }{ q^{b(r-1) - 5r^3}} $, where $\mathcal{D}_a$ is the set of all dangerous sets of Type 1 and size~$a$, and $\mathcal{D}_b$ is the set of all dangerous sets of Type 2 and size~$b$.
\end{enumerate}
\end{lemma}

\begin{proof} Let $V_1 = V$ be the same vertex set as $\mathcal{H}$, and let $\mathbb{H}_1=(V_1,\mathcal{L}_1)$ be a random sub-hypergraph of $\mathcal{H}$ where every line in $\mathcal{L}$ is taken independently with probability 
${\lambda}/{q}$.

Since $\mathbb{H}_1$ is a subgraph of $\mathcal{H}$, any two vertices are in at most one line, so $\mathbb{H}_1$ always satisfies \ref{H1:a}. 
We will show that $\mathbb{H}_1$ satisfies \whp \ref{H1:b} and satisfies \ref{H1:c} with probability at least $1/2$. 
Together this implies that $\mathbb{H}_1$ satisfies \ref{H1:a}-\ref{H1:c} with probability at least $1-\frac{1}{2}-o(1)$, establishing the existence of a hypergraph $\mathcal{H}_1$ that satisfies \ref{H1:a}-\ref{H1:c}. 

\medskip

\textbf{\ref{H1:b}:}  Observe for fixed $v \in V_1$, $\deg_{\mathbb{H}_1}(v) \sim \bin(q, \frac{\lambda}{q})$ and the expected value $\E(\deg_{\mathbb{H}_1}(v))= \lambda.$ So by the Chernoff bound with $\varepsilon = \frac{1}{2}$, 
\[
 \Pr \Big( |\deg_{\mathbb{H}_1}(v)-\lambda| \geq \frac{\lambda}{2} \Big) \leq 2  \exp \left\{ -\frac{\lambda}{12} \right\}. 
\]
Thus by the union bound the probability that there exists some $v \in V_1$ with $\deg_{\mathbb{H}_1}(v) \notin \sqbs{\frac{\lambda}{2} , \frac{3\lambda}{2}}$ is at most
\[
q^2 \cdot 2  \exp \left\{-\frac{\lambda}{12}\right\} = 2 \exp \left\{2 \log q-\frac{\lambda}{12}\right\} = o(1),
\]
since $\lambda\gg\log q$.

\medskip

\textbf{\ref{H1:c}:} In order to show  we have both $|\mathcal{D}_a|  \leq  \frac{4\lambda^{a^2/2} }{ q^{(a^2-5a)/2}}$ and $| \mathcal{D}_b| \leq \frac{4\lambda^{br } }{ q^{b(r-1) - 5r^3}} $ with probability at least $\frac{1}{2}$, we begin by counting the number of dangerous subsets of each type. Clearly the number of Type 1 dangerous subsets is at most ${q^2 \choose a}$ and each of them contains $\binom{a}{2}$ lines. To count the number of Type 2 dangerous subsets first we choose a line $L$ out of $q^2$ lines in~$\mathcal{L}$. Next we choose $b-r$ points in $L$ having $\binom{q}{b-r}$ choices, and finally, we choose the remaining $r$ points not in $L$. Thus, the number of configurations of Type~2 is at most $\binom{q^2}{1} \cdot \binom{q}{b-r} \cdot q^{2r}$. Now we bound the number of lines in a dangerous set of Type 2. First there is the line $L$ containing $|S|-r$ points. Then, every pair of points $u,v$, where $u \in L \cap S$ and $v \in R$, must be contained in some line $L_{u,v}$ in $\mathcal{L}$. No line $L_{u,v}$ can contain more than one vertex in $L$, but it is possible for some line $L_{u,v}$ to contain multiple vertices in $R$. However for $v, v' \in R$, if $L_{u,v}$ contains $v'$ then no other line can contain both $v, v'$. Thus, the number of lines of the form $L_{u,v}$ such that $|L_{u,v} \cap R|\ge 2$ is at most $\binom{r}{2}$. Now since each of the $r(b-r)$ many pairs $\{u,v\}$ must be covered by some $L_{u,v}$, and the number of lines covering multiple pairs is at most $\binom{r}{2}$, and no line covers more than $r$ many pairs, the total number of lines $L_{u,v}$ is at least $r(b-r) - r\binom{r}{2} \ge br - 2r^3$. 

By the linearity of expectation, we now compute 
\[
\E(| \mathcal{D}_a|) 
\leq \binom{q^2}{a}  \cdot \left( \frac{\lambda}{q} \right)^{\binom{a}{2}} \leq q^{2a} \cdot \left( \frac{\lambda}{q} \right)^{\binom{a}{2}} \le  \frac{\lambda^{a^2 /2 }}{ q^{(a^2-5a)/2}}
\]
and
\begin{align*}
\E(|\mathcal{D}_b|) 
&\leq  \binom{q^2}{1} \cdot \binom{q}{b-r} \cdot q^{2r} \cdot \left( \frac{\lambda}{q} \right)^{br - 2r^3}\\
&\leq   q^{2} \cdot q^{b-r} \cdot q^{2r} \cdot \left( \frac{\lambda}{q} \right)^{br - 2r^3} \le \frac{\lambda^{br } }{ q^{b(r-1) - 5r^3}}.
\end{align*}
Thus, the Markov bound yields 
\[
\Pr \left( |\mathcal{D}_a| \ge   \frac{4\lambda^{a^2 /2 }}{ q^{(a^2-5a)/2}} \right)
\le \Pr \left( |\mathcal{D}_a| \ge 4 \E ( |\mathcal{D}_a|) \right) \le \frac{1}{4}
\]
and 
\[
\Pr \left( | \mathcal{D}_b| \ge    \frac{4\lambda^{br } }{ q^{b(r-1) - 5r^3}}\right)
\le \Pr \left( |\mathcal{D}_b| \ge 4 \E ( | \mathcal{D}_b|) \right) \le \frac{1}{4},
\]
and finally
\[
\Pr \left( |\mathcal{D}_a| \le   \frac{4\lambda^{a^2 /2 }}{ q^{(a^2-5a)/2}} \mbox{ and } | \mathcal{D}_b| \le    \frac{4\lambda^{br } }{ q^{b(r-1) - 5r^3}}\right) \ge 1 - \frac 14 - \frac 14 = \frac 12,
\]
as required.
\end{proof}

%%%%%%%%%%%%%%%%%%%%%%%%%%%%%%%%%%%%%%%%%%%%%%%%%%%%%%%%%%%%
%%%%%%%%%%%%%%%%%%%%  The Graph G %%%%%%%%%%%%%%%%%%%%%%%%%%
%%%%%%%%%%%%%%%%%%%%%%%%%%%%%%%%%%%%%%%%%%%%%%%%%%%%%%%%%%%%

\subsection{The graph $G_1$}\label{sec:G1}
Based upon the hypergraph $\mathcal{H}_1$ established in the previous section, we will construct a graph $G_1$ with the following properties. 
\begin{lemma} \label{lem:3}
Let $r\ge 1$ and $s$ be sufficiently large constant.
Let  $q$ be a sufficiently large prime, $q \ge \lambda \gg \log q$, $ 1 \ge p \gg (\log q) / \lambda$, and $\alpha \ge (10 s \log s)q/p$. Furthermore, let $a$ be a positive constant and $b=\left\lceil(s+r) /{ \binom{a-1}{2} }\right\rceil +r$. Then, there exists a graph $G_1$ of order $q^2$ such that:
\begin{enumerate}[label=$({\textup{G}}_{1}\textup{\alph*})$]
\item\label{G1:a} $\alpha_s(G_1) < \alpha$;
\item\label{G1:b} For every vertex $v$ in $G_1$, $ \deg_{G_1}(v)= \Theta( \lambda q p^2)$;
\item\label{G1:c} $G_1$ has at most  $8\left(\frac{\lambda^{a^2 /2} p^{a^2 -a}}{ q^{(a^2-5a)/2}} + \frac{\lambda^{br } p^{(2r+1)b - 4r^3}}{ q^{b(r-1) -5r^3}}\right)$ copies of  $K_{s+r}$.
\end{enumerate}
\end{lemma}

\begin{proof}  Fix a hypergraph $\mathcal{H}_1=(V_1,\mathcal{L}_1)$ as established by Lemma~\ref{lem:1}. Construct the random graph $\mathbb{G}_1=(V_1,E)$ as follows. For every $L\in \mathcal{L}_1$, let $\chi_L : L \to  {[s+1]}$ be a random partition of the vertices of $L$ into $ {s+1}$ classes, where for every $v \in L$, 
\[
\Pr(\chi(v) = i) = 
\begin{cases}
{p}/{s} & \text{for } 1\le i\le s,\\
1-p & \text{for } i=s+1,
\end{cases}
\]
and $\chi(v)$ is assigned independently from other vertices.
 If $\chi_L(v) =s+1$, then we say that $v$ is {\em $L$-isolated}. Let $\{x,y\} \in E$ if $\{x,y\} \subseteq L$ for some $L \in \mathcal{L}_1$ and $\chi_L(x),  \chi_L(y)$ are distinct and neither $x$ nor $y$ is $L$-isolated. Thus for every $L \in \mathcal{L}_1$, $\mathbb{G}_1[L]$ consists of a set of isolated vertices (the $L$-isolated vertices) together with a complete $s$-partite graph with vertex partition $L = \chi^{-1}_L (1) \cup \chi^{-1}_L (2) \cup \dots \cup \chi^{-1}_L (s)$ (where the classes need not have the same size and the unlikely event that a class is empty is permitted). Observe that not only are $\mathbb{G}_1[L]$ and $\mathbb{G}_1[L']$ edge disjoint for distinct $L, L' \in \mathcal{L}_1$, but also that the partitions for $L$ and $L'$ were determined independently.

We will show that $\mathbb{G}_1$ satisfies \whp \ref{G1:a} and \ref{G1:b} and satisfies \ref{G1:c} with probability at least $1/2$. 

\medskip

\textbf{\ref{G1:a}:} First fix $C \in {V_1 \choose \alpha}$. We will bound the probability that $\mathbb{G}_1[C] \not \supseteq K_s$. For a fixed $L$, the probability that one of the classes $\chi_L^{-1}(1), \dots , \chi_L^{-1}(s)$ contains no element of $C$ is at most $s\of{1-\frac ps}^{|L \cap C|} \le s e^{-\frac {p}{s} |L \cap C|}$. Note that 
\[
\sum_{L \in \mathcal{L}_1} |L \cap C| \ge \frac 12 \lambda |C| =  \frac 12 \lambda \alpha,
\] 
since each point is in at least $\frac 12 \lambda$ lines due to condition~\ref{H1:b}. Now since $\chi_L$ and $\chi_{L'}$ are chosen independently for $L \neq L'$ we get, 
\[
\Pr\Big( K_s \not \subseteq \mathbb{G}_1[C] \Big) \le \prod_{L \in \mathcal{L}_1}\Pr \Big( K_s \not \subseteq \mathbb{G}_1[L \cap C] \Big) \le s^{|\mathcal{L}_1|} \exp\left\{ -\frac ps \sum_{L \in \mathcal{L}_1} |L \cap C| \right\},
\]
and since $|\mathcal{L}_1| \le 3\lambda q$ by~\ref{H1:b}, 
\[
\Pr\Big( K_s \not \subseteq \mathbb{G}_1[C] \Big) \le \exp \left\{(3\log s) \lambda q  - \frac{1}{2s} \lambda \alpha p \right\}.
\]

So by the union bound, the probability that there exists a set $C$ of $\alpha$ vertices in $\mathbb{G}_1$ that contains no $K_s$ is at most
\begin{align*}
&q^{2 \alpha} \exp \left\{(3\log s) \lambda q  - \frac{1}{2s} \lambda \alpha p  \right\}\le   \exp \left\{2\alpha  \log q +(3\log s) \lambda q  - \frac{1}{2s} \lambda \alpha p \right\}=o(1),
\end{align*}
since $ p \gg  (\log q) / \lambda$ and $\alpha \ge (10 s \log s)q/p$.

Thus, \whp $\alpha_s(\mathbb{G}_1) < \alpha$.

\medskip

\textbf{\ref{G1:b}:} Observe that for a fixed $L\in \mathcal{L}_1$, the number of non-$L$-isolated vertices $|\chi_L^{-1}([s])|$ is distributed as $\bin(q, p)$ which has expectation $pq $ so by the Chernoff bound with $\varepsilon = 1/2$ get that
\[ 
\Pr\of{ \left| |\chi_L^{-1}([s])| -pq \right| > \frac 12 pq} \le 2\exp \left\{-\frac{pq}{12} \right\}
\]
and so by the union bound, the probability that there exists some $L$ such that\linebreak $\left| |\chi_L^{-1}([s])| -pq \right| > \frac 12 pq$ is at most 
\[
q^2 \cdot  2\exp \left\{-\frac{pq}{12} \right\} = 2\exp \left\{2\log q-\frac{pq}{12} \right\} = o(1),
\] 
since $p\gg (\log q) / \lambda$ and $\lambda \le q$.
Thus, \whp every line has between $\frac 12 pq$ and $\frac 32 pq$ many non-$L$-isolated vertices. 

Now for fixed $v$, let $X_v$ be the number of lines $L$ in which $v$ is non-$L$-isolated. $X_v$ is distributed as $\bin(\deg_{\mathcal{H}_1} (v), p)$ which has expectation $\deg_{\mathcal{H}_1} (v)p  \ge \lambda p/2$. Now by the Chernoff bound with $\varepsilon = 1/2$ we get 
\[
\Pr\of{|X_v - \deg_{\mathcal{H}_1} (v)p| > \frac 12 \deg_{\mathcal{H}_1} (v)p} \le 2 \exp \left\{-\frac{\deg_{\mathcal{H}_1} (v)p}{12} \right\} \le 2 \exp \left\{-\frac{\lambda p}{24} \right\}
\]
and so by the union bound, the probability that there exists some point $v$ with $|X_v - \deg_{\mathcal{H}_1} (v)p| > \frac 12 \deg_{\mathcal{H}_1} (v)p$ is at most 
\[
q^2 \cdot 2 \exp \left\{-\frac{\lambda p}{24} \right\} = 2 \exp \left\{2\log q-\frac{\lambda p}{24} \right\} = o(1),
\]
since $p \gg (\log q) / \lambda$.
So \whp for every $v$, $v$ is non-$L$-isolated for some number of lines $L$ that is at least  $\frac 12 \deg_{\mathcal{H}_1} (v)p \ge \frac 14 \lambda p$ and at most $\frac 32 \deg_{\mathcal{H}_1} (v)p \le \frac 94 \lambda p$.

Now assume for each $L$ and each $v \in L$ we have revealed whether $v$ is $L$-isolated, but we have not revealed $\chi_L(v)$ when $v$ is non-$L$-isolated. When we do reveal values $\chi_L(v)$ to form the graph $\mathbb{G}_1$, we have for each non-$L$-isolated vertex $v$ that
$\deg_{\mathbb{G}_1[L]}(v) \sim \bin(|\chi_L^{-1}([s])|-1, \frac{s-1}{s})$. Thus,
$\E(\deg_{\mathbb{G}_1[L]}(v)) = (|\chi_L^{-1}([s])|-1)(s-1)/s \ge (\frac 12 pq-1)(s-1)/s$  and the Chernoff bound with $\varepsilon = 1/2$ tells us that 
\begin{align*}
\Pr\of{|\deg_{\mathbb{G}_1[L]}(v) - \E(\deg_{\mathbb{G}_1[L]}(v))| \ge \frac 12 \E(\deg_{\mathbb{G}_1[L]}(v))} &\le 2\exp\left\{- \frac{\E(\deg_{\mathbb{G}_1[L]}(v))}{12} \right\}\\
& \le 2\exp\left\{- \frac{(\frac 12 pq-1)(s-1)/s}{12} \right\}
\end{align*}
so by the union bound, the probability that there is a vertex $v$ with $|\deg_{\mathbb{G}_1[L]}(v) - \E(\deg_{\mathbb{G}_1[L]}(v))| \ge \frac 12 \E(\deg_{\mathbb{G}_1[L]}(v))$ is at most 
\[
q^2 \cdot 2\exp\left\{- \frac{(\frac 12 pq-1)(s-1)/s}{12} \right\} = o(1),
\] 
since $p \gg (\log q) / \lambda$ and $\lambda \le q$. Thus \whp for every vertex $v$ and line $L$ for which $v$ is non-$L$-isolated we have that $\deg_{\mathbb{G}_1[L]}(v)$ is at least $\frac 12 (|\chi_L^{-1}([s])|-1)(s-1)/s \ge \frac 18 pq$ and at most $\frac 32 (|\chi_L^{-1}([s])|-1)(s-1)/s \le \frac 94 pq$. 

Thus, for each vertex $v$, \whp its degree in $\mathbb{G}_1$ is $\Theta( \lambda q p^2)$.

\medskip

\textbf{\ref{G1:c}:} Recall that $a$ is some positive integer and $b=\left\lceil(s+r)/{ \binom{a-1}{2} }\right\rceil +r.$
First we show that every copy of $K_{s+r}$ in $\mathbb{G}_1$ contains a dangerous subset in $\mathcal{D}_a \cup \mathcal{D}_b$.
Let $K$ be any copy of $K_{s+r}$ in $\mathbb{G}_1$. Clearly if $K$ contains a subset of $a$ points in general position then such subset is in $\mathcal{D}_a$. Assume that $K$ contains at most $a-1$ points in general position. These points can be covered by at most $\binom{a-1}{2}$ lines. In fact, all of the  $s+r$ points must belong to those lines. Thus, there is a line $L$ with at least $\lceil (s+r)/\binom{a-1}{2}\rceil =b-r$ points from $K$. Moreover, since each line can contain at most $s$ points from $K$ there is a set $R$ of $r$ additional points in~$K$  that do not belong to~$L$. Hence, set $L\cup R$ is in $\mathcal{D}_b$. 

By \ref{H1:c},  $|\mathcal{D}_a|  \leq  \frac{4\lambda^{a^2 / 2} }{ q^{(a^2-5a)/2}}$, and $| \mathcal{D}_b| \leq \frac{4\lambda^{br } }{ q^{b(r-1) - 5r^3}} $. We will show that \whp none of these dangerous sets gives rise to a $K_{s+r}$ in $\mathbb{G}_1$. In order for such dangerous set to give a $K_{s+r}$ in $\mathbb{G}_1$, none of the vertices can be $L$-isolated in any of the lines in the dangerous set. For a Type 1 dangerous set, there are $\binom{a}{2}$ lines, each containing $2$ vertices, so the probability that no vertex $v$ is $L$-isolated for $L$ containing $v$ is $p^{2 \binom{a}{2}} = p^{a^2 - a}$. Thus, the expected number of copies of $K_{s+r}$ that arise from Type 1 dangerous sets  is at most
\[
\frac{4\lambda^{a^2/2}}{ q^{(a^2-5a)/2}}\cdot p^{a^2-a}.
\]
Now for dangerous sets of Type 2, the probability that a dangerous set in $\mathcal{D}_b$ gives rise to a $K_{s+r}$ is $p^x$, where $x$ is the number of point-line incidences there are within the dangerous set. We observed before that each Type~2 dangerous set consists of a line~$L$ with $b-r$ points and a set $R$ of $r$~points and at least $r(b-r) - r\binom{r}{2}$ lines containing one point from $L$ and at least one point from $R$. Thus, the number of point-line incidences is at least
\[
(b-r) + 2 \left(r(b-r) - r\binom{r}{2} \right) \ge (2r+1)b - 4r^3.
\]
Therefore, the expected number of copies of $K_{s+r}$ arising from Type 2 dangerous sets is at most
\[
\frac{4\lambda^{br}}{ q^{b(r-1) - 5r^3}} \cdot p^{(2r+1)b - 4r^3}.
\]

Thus by linearity of expectation the total number of copies of $K_{s+r}$ has expectation at most
\[
\frac{4\lambda^{a^2/2} p^{a^2-a}}{ q^{(a^2-5a)/2}} + \frac{4\lambda^{br} p^{(2r+1)b - 4r^3}}{ q^{b(r-1) - 5r^3}}
\]
and so we are done by the Markov bound applied with $\zeta= 2$. 
\end{proof}

%%%%%%%%%%%%%%%%%%%%%%%%%%%%%%%%%%%%%%%%%%%%%%%%%%%%%%%%%%%%
%%%%%%%%%%%%%%%%%%%%  Lower bound  %%%%%%%%%%%%%%%%%%%%
%%%%%%%%%%%%%%%%%%%%%%%%%%%%%%%%%%%%%%%%%%%%%%%%%%%%%%%%%%%%

\subsection{Deriving the lower bound for small values of $\delta$} \label{sec:LB1}

\begin{theorem}\label{thm:main1}
Let $r\ge 1$, $\varepsilon>0$, and $\frac 12 < \delta \le \frac 12 + \frac{1}{2(2r+1)}$. Then for all sufficiently large $s$, we have 
\begin{equation}\label{thm:main1:eq1}
\RT_s(n, K_{s+r}, n^\delta) = \Omega\of{n^{2 - (1-\delta)/r  - \varepsilon}}.
\end{equation}
Furthermore, if $r\ge 2$, then there exists a positive constant $C=C(s)$ such that for all sufficiently large $s$
\begin{equation}\label{thm:main1:eq2}
\RT_s(n, K_{s+r}, C\sqrt{n}) = \Omega\of{n^{2 - 1/(2r)  - \varepsilon}}.
\end{equation}
\end{theorem}

\begin{proof}
We will apply Lemma~\ref{lem:3} after discussing how to set parameters. First we prove~\eqref{thm:main1:eq1} assuming that $\frac 12 < \delta \le \frac 12 + \frac{1}{2(2r+1)}$.

Fix $r\ge 1$ and $\varepsilon>0$. It is known that for large $x$ there exists a prime number between $x$ and $x(1+o(1))$ (see, e.g.,~\cite{BHP}). Hence, for large $n$ there is a prime number $q$ such that $\sqrt{n} \le q \le (1+o(1))\sqrt{n}$. Set $\alpha = n^\delta$ and $p = \kappa q^{-(2 \delta -1)}$, where $\kappa = 20s\log s$. 
We will show that the above parameters satisfy all assumptions of Lemma~\ref{lem:3} implying the existence of a graph $G_1$ of order $q^2$ satisfying \ref{G1:a}-\ref{G1:c}.

First observe that $(10 s \log s)q/p = q^{2\delta} / 2 \le n^{\delta} = \alpha$, as required by Lemma~\ref{lem:3}. Thus, due to~\ref{G1:a} the $s$-independence number of $G_1$ is less than $\alpha$. 

Set $a =20r$, and $b=\left\lceil(s+r)/ \binom{a-1}{2} \right\rceil +r$. We will assume that $s$, and consequently $b $, is large enough such that for example
\begin{equation}\label{thm:main1:eps}
\varepsilon > 5r^2/b.
\end{equation}
Furthermore, let
\[
\lambda = q^{1-1/r + (2\delta -1)(2+1/r) - 10r^2/b}.
\]
Observe that 
\[
\lambda p = \kappa q^{1-1/r + (2\delta -1)(1+1/r) - 10r^2/b} \gg \log q,
\]
since $\delta>1/2$ and $b$ is sufficiently large. This implies that $p \gg (\log q) / \lambda$, as required in Lemma~\ref{lem:3}. Clearly this also implies that $\lambda \gg \log q$. Also note that since $\delta \le \frac 12 + \frac{1}{2(2r+1)}$,
\[
1- \frac 1r + (2\delta -1)\of{2+ \frac 1r} - \frac{10r^2}{b} \le 1 - \frac 1r + \frac{1}{2r+1}\of{2+ \frac 1r} - \frac{10r^2}{b}=1 - \frac{10r^2}{b}
\] 
so $\lambda < q$ and hence all assumptions of Lemma~\ref{lem:3} are satisfied.

Now we show that $G_1$ is $K_{s+r}$-free. By~\ref{G1:c} The number of copies of $K_{s+r}$ is at most 
\begin{equation} \label{nums+r}
8 \left(\frac{\lambda^{a^2 /2} p^{a^2 -a}}{ q^{(a^2-5a)/2}} + \frac{\lambda^{br } p^{(2r+1)b - 4r^3}}{ q^{b(r-1) -5r^3}}\right).
\end{equation}
The order of magnitude of the first term in~\eqref{nums+r} is 
\begin{align*}
\frac{q^{\left(1-1/r + (2\delta -1)(2+1/r) - 10r^2/b\right)  a^2/2} \cdot q^{(1-2\delta)(a^2-a)}}{q^{(a^2-5a)/2}}
&= q^{(a^2/2) \cdot \left( -(2-2\delta)/r - 10r^2/b \right) + (a/2) \cdot \left( 2(2\delta-1) +5\right)} \\
& \le q^{-a^2/(2r+1) + 7a/2} = o(1),
\end{align*}
where in the last line we used $\delta \le \frac 12 + \frac{1}{2(2r+1)}$ and $- 10r^2/b < 0$, and the fact that $a=20r$. Now the second term of \eqref{nums+r} has order of magnitude at most
\[
\frac{q^{\left( 1-1/r + (2\delta -1)(2+1/r) - 10r^2/b\right) \cdot br} q^{(1-2\delta)\of{(2r+1)b-4r^3}} }{ q^{b(r-1) -5r^3}}
=q^{-5r^3 + 4r^3(2\delta-1)} \le q^{-r^3} = o(1).
\]

Now let $G$ be any induced subgraph of $G_1$ of order $n$. Clearly, $G$ is $K_{s+r}$-free with $\alpha_s(G) < n^{\delta}$. Furthermore, 
\begin{align*}
|E(G)| &\ge |E(G_1)| - |V(G_1) - V(G)| \cdot \Delta(G_1)\\ 
& \ge |V(G_1)|\cdot \delta(G_1) - o(1) \cdot \Delta(G_1)  = \Omega\of{ q^2 \cdot \lambda q p^2},
\end{align*}
by~\ref{G1:b}. Finally observe that 
\begin{align*}
q^2 \cdot \lambda q p^2 = \lambda q^3 p^2 &\ge q^{1-1/r + (2\delta -1)(2+1/r) - 10r^2/b} \cdot q^3 \cdot \of{ q^{1-2\delta}}^2\\
& \ge q^{4-2(1-\delta)/r - 10r^2/b} = n^{2-(1-\delta)/r - 5r^2/b } >  n^{2-(1-\delta)/r - \varepsilon},
\end{align*}
because of~\eqref{thm:main1:eps}. Thus, $|E(G)| = \Omega(n^{2-(1-\delta)/r - \varepsilon})$ yielding the lower bound in~\eqref{thm:main1:eq1}.

The proof of~\eqref{thm:main1:eq2} is very similar. Assume that $r\ge 2$ and let $\delta=1/2$, $p=1$, and  $\alpha = C \sqrt{n}$ with $C=20s(\log s)$. Other parameters are the same. Then, as in the previous case the assumptions of Lemma~\ref{lem:3} hold.
\end{proof}

\section{Lower bound on $\RT_s(n, K_{s+r}, n^\delta)$ for intermediate values of  $\delta$}\label{sec:main2}

Recall that in Section~\ref{sec:main1} we started with the affine plane and made it sparser by taking edges with probability $\lambda / q$. In Section~\ref{sec:LB1}, to optimize our result we set 
\[
\lambda = q^{1-1/r + (2\delta -1)(2+1/r) - 10r^2/b},
\]
so when $\delta = \frac 12 + \frac{1}{2(2r+1)}$, we are setting $\lambda$ to be nearly $q$ (since $r^2 / b$ is small), which is about as large as $\lambda$ can possibly be (and then we are making our hypergraph $\mathcal{H}_1$ nearly as dense as the original affine plane). Thus it makes sense that if $\delta$ is bigger than $\frac 12 + \frac{1}{2(2r+1)}$ then we no longer want a sparser version of the affine plane, but a denser version. Therefore in this section we will construct a denser hypergraph $\mathcal{H}_2$ by keeping all of the edges and ``eliminating" some vertices. That is the key difference between Sections~\ref{sec:main1} and~\ref{sec:main2}. Otherwise the proofs are quite similar.

\subsection{The hypergraph $\mathcal{H}_2$} \label{sec:H2}

In this section we establish the existence of a hypergraph $\mathcal{H}_2$ with certain properties.

\begin{lemma} \label{lem:4}
Let $q$ be a sufficiently large prime and $\log q \ll \lambda \le q$. 
Then, there exists a $q$-regular hypergraph $\mathcal{H}_2 = (V_2, \mathcal{L}_2)$ of order $\lambda q (1+o(1))$ such that:
\begin{enumerate}[label=$({\textup{H}}_{2}\textup{\alph*})$]
\item\label{H2:a} Any two vertices are contained in at most one hyperedge;
\item\label{H2:b} For every $L \in \mathcal{L}_2$, $\frac{\lambda}{2} \le |L| \leq \frac{3\lambda}{2}$; 
\item\label{H2:c} $|\mathcal{D}_a|  \leq  4\lambda^a q^a$ and $| \mathcal{D}_b| \leq 4\lambda^b q^{r+2}$, where $\mathcal{D}_a$ is the set of all dangerous sets of Type~1 and size~$a$ and $\mathcal{D}_b$ is the set of all dangerous sets of Type~2 and size~$b$.
\end{enumerate}
\end{lemma}

\begin{proof}
Starting with the hypergraph $\mathcal{H}$, we randomly ``eliminate" some points. For each $v\in V$ we randomly (and independently from other vertices) choose to eliminate $v$  with probability $1 -  \lambda / q$. Say $X$ is the set of vertices chosen for elimination. By ``elimination", we mean that we will form a new hypergraph $\mathbb{H}_2$ with vertex set $V_2 = V \setminus X$, and edge set 
\[
\mathcal{L}_2 = \{L \setminus X: L \in \mathcal{L} \}.
\]

First we will prove that \whp $\mathbb{H}_2$ has $\lambda q (1+o(1))$ vertices. The number of vertices~$V$ is distributed as $\bin(q^2, \lambda / q)$ which has expectation $\lambda q$. By the Chernoff bound with $\varepsilon = (\lambda q)^{-1/4}$,
\[
\Pr\of{|V - \lambda q| > (\lambda q)^{3/4}} \le 2 \exp \left\{ - \frac{(\lambda q )^{1/2}}{3}\right\} = o(1).
\]

Now by the construction of $\mathbb{H}_2$ and the properties of $\mathcal{H}$, two vertices are in at most one line, so $\mathbb{H}_2$ always satisfies \ref{H2:a}. 
We will show that $\mathbb{H}_2$ satisfies \whp \ref{H2:b} and satisfies \ref{H2:c} with probability at least $1/2$. 
Together this implies $\mathbb{H}_2$ satisfies \ref{H2:a}-\ref{H2:c} with probability at least $1-\frac{1}{2}-o(1)$, establishing the existence of a hypergraph $\mathcal{H}_2$ that satisfies \ref{H2:a}-\ref{H2:c}. 

\medskip

\textbf{\ref{H2:b}:}  Observe that for fixed $L \in \mathcal{L}_2$, $|L| \sim \bin(q, \frac{\lambda}{q})$ and $\E(|L|)= \lambda.$ So by the Chernoff bound with $\varepsilon = \frac{1}{2}$, 
\[
 \Pr \Big( \left||L|-\lambda\right| \geq \frac{\lambda}{2} \Big) \leq 2  \exp \left\{ -\frac{\lambda}{12} \right\}. 
\]
Thus by the union bound the probability that there exists some $L \in \mathcal{L}_2$ with $|L| \notin \sqbs{\frac{\lambda}{2} , \frac{3\lambda}{2}}$ is at most
\[
q^2 \cdot 2  \exp \left\{-\frac{\lambda}{12}\right\} = 2 \exp \left\{2 \log q-\frac{\lambda}{12}\right\} = o(1).
\]

\medskip

\textbf{\ref{H2:c}:} In order to show that both $|\mathcal{D}_a|  \leq  4\lambda^{a} q^a$ and $|\mathcal{D}_b| \leq 4\lambda^b q^{r+2}$ with probability at least $\frac{1}{2}$, we recall from Section~\ref{sec:H1} the number of dangerous subsets of each type. The number of Type~1 dangerous subsets is at most ${q^2 \choose a}$, and the number of Type~2 dangerous subsets is at most $\binom{q^2}{1} \cdot \binom{q}{b-r} \cdot q^{2r}$. In order for $\mathbb{H}_2$ to inherit a dangerous set, none of its vertices can be eliminated. 
By the linearity of expectation, we now compute 
\[
\E(| \mathcal{D}_a|) \leq \binom{q^2}{a}  \cdot \left( \frac{\lambda}{q} \right)^{a} \leq \lambda^a q^a
\]
and
\[
\E(|\mathcal{D}_b|) 
\leq  \binom{q^2}{1} \cdot \binom{q}{b-r} \cdot q^{2r} \cdot \left( \frac{\lambda}{q} \right)^{b}
\leq   q^{2} \cdot q^{b-r} \cdot q^{2r} \cdot \left( \frac{\lambda}{q} \right)^{b} = \lambda^b q^{r+2},
\]
and we are done by the Markov bound applied twice with $\zeta=4$.
\end{proof}

\subsection{The graph $G_2$}\label{sec:G2}

Based upon the hypergraph $\mathcal{H}_2$ established in the previous section, we will construct a graph $G_2$ with the following properties. 

\begin{lemma} \label{lem:5}
Let $r\ge 1$ and $s$ be sufficiently large constant.
Let  $q$ be a sufficiently large prime, $q \ge \lambda \gg \log q $, $1 \ge p \gg (\log q) / \lambda$, and $\alpha \ge (10 s \log s)q/p$. Furthermore, let $a$ be a positive constant and $b=\left\lceil(s+r)/\binom{a-1}{2}\right\rceil +r$. Then, there exists a graph $G_2$ with $\lambda q (1+o(1))$ vertices such that:
\begin{enumerate}[label=$({\textup{G}}_{2}\textup{\alph*})$]
\item\label{G2:a} $\alpha_s(G_2) < \alpha$;
\item\label{G2:b} For every vertex $v$ in $G_2$, $ \deg_{G_2}(v)= \Theta( \lambda q p^2)$;
\item\label{G2:c} $G_2$ has at most  $8\left(\lambda^a q^a p^{a^2 -a} +  \lambda^b q^{r+2} p^{(2r+1)b - 4r^3}\right)$ copies of  $K_{s+r}$.
\end{enumerate}
\end{lemma}

\begin{proof}
Starting with the hypergraph $\mathcal{H}_2=(V_2,\mathcal{L}_2)$, we form the random graph $\mathbb{G}_2=(V_2,E)$ as follows. For every $L\in \mathcal{L}_2$, let $\chi_L : L \to  {[s+1]}$ be a random partition of the vertices of $L$ into $ {s+1}$ classes, where for every $v \in L$,
\[
\Pr(\chi(v) = i) = 
\begin{cases}
{p}/{s} & \text{for } 1\le i\le s,\\
1-p & \text{for } i=s+1,
\end{cases}
\]
and $\chi(v)$ is assigned independently from other vertices.
If $\chi_L(v) =s+1$, then we say that $v$ is {\em $L$-isolated}. Let $\{x,y\} \in E$ if $\{x,y\} \subseteq L$ for some $L \in \mathcal{L}_2$ and $\chi_L(x),  \chi_L(y)$ are distinct and neither $x$ nor $y$ is $L$-isolated. Thus for every $L \in \mathcal{L}_2$, $\mathbb{G}_2[L]$ consists of a set of isolated vertices (the $L$-isolated vertices) together with a complete $s$-partite graph with vertex partition $L = \chi^{-1}_L (1) \cup \chi^{-1}_L (2) \cup \dots \cup \chi^{-1}_L (s)$.

We will show that $\mathbb{G}_2$ satisfies \whp \ref{G2:a} and \ref{G2:b} and satisfies \ref{G2:c} with probability at least $1/2$. 

\medskip

\textbf{\ref{G2:a}:} First fix $C \in {V_2 \choose \alpha}$. We will bound the probability that $\mathbb{G}_2[C] \not \supseteq K_s$. For a fixed $L$, the probability that one of the classes $\chi_L^{-1}(1), \dots, \chi_L^{-1}(s)$ contains no element of $C$ is at most $s\of{1-\frac ps}^{|L \cap C|} \le s e^{-\frac {p}{s} |L \cap C|}$. Note that 
\[
\sum_{L \in \mathcal{L}_2} |L \cap C| = q |C| =   \alpha q,
\] 
since each point is in  $q$ lines. Now since $\chi_L$ and $\chi_{L'}$ are chosen independently for $L \neq L'$ we get, 
\begin{align*}
\Pr\Big( K_s \not \subseteq \mathbb{G}_2[C] \Big) &\le \prod_{L \in \mathcal{L}_2}\Pr \Big( K_s \not \subseteq \mathbb{G}_2[L \cap C] \Big)\\
& \le s^{q^2} \exp\left\{ -\frac ps \sum_{L \in \mathcal{L}_2} |L \cap C| \right\} \le \exp \left\{(\log s)  q^2  - \frac{1}{s} \alpha q p \right\}
\end{align*} 

So by the union bound, the probability that there exists a set $C$ of $\alpha$ vertices in $\mathbb{G}_2$ that contains no $K_s$ is at most
\[
q^{2 \alpha} \cdot \exp \left\{(\log s)  q^2  - \frac{1}{s} \alpha q p \right\} \le   \exp \left\{2\alpha  \log q +(\log s)  q^2  - \frac{1}{s} \alpha q p \right\}=o(1),
\]
since $p \gg (\log q) / \lambda$, $\lambda \le q$, and  $\alpha \ge (10 s \log s) q/p$. 

Thus, \whp $\alpha_s(\mathbb{G}_2) < \alpha$.

\medskip

\textbf{\ref{G2:b}:} Observe that for a fixed $L\in \mathcal{L}_2$, the number of non-$L$-isolated vertices $|\chi_L^{-1}([s])|$ is distributed as $\bin(|L|, p)$ which has expectation $p|L| $ so by the Chernoff bound with $\varepsilon = 1/2$ we get that
\[
\Pr\of{ \left| |\chi_L^{-1}([s])| -p|L| \right| > \frac 12 p|L|} \le 2\exp \left\{-\frac{p|L|}{12} \right\}\le 2\exp \left\{-\frac{p\lambda}{24} \right\}
\]
and so by the union bound, the probability that there exists some $L$ such that\linebreak 
$\left| |\chi_L^{-1}([s])| -p|L| \right| >  \frac 12 p |L|$ is at most 
\[
q^2 \cdot  2\exp \left\{-\frac{p \lambda}{24} \right\} = o(1),
\]
since $p\gg (\log q) / \lambda$.
Thus, \whp every line $L$ has at least $\frac 12 p|L| \ge \frac 14 p \lambda$ and at most $\frac 32 p|L| \le \frac 94 p \lambda$ many non-$L$-isolated vertices. 

Now for fixed $v$, let $X_v$ be the number of lines $L$ in which $v$ is non-$L$-isolated. $X_v$ is distributed as $\bin(q , p)$ which has expectation $qp $. Now by the Chernoff bound with $\varepsilon = 1/2$ we get 
\[
\Pr\of{|X_v - qp| > \frac 12 qp} \le 2 \exp \left\{-\frac{qp}{12} \right\}
\] 
and so by the union bound, the probability that there exists some point $v$ with $|X_v - qp| > \frac 12 qp$ is at most 
\[
q^2 \cdot 2 \exp \left\{-\frac{q p}{12} \right\} = o(1),
\]
since $ p \gg (\log q) / \lambda$ and $\lambda \le q$.
So \whp for every $v$, $v$ is non-$L$-isolated for some number of lines $L$ that is between  $\frac 12 qp $ and at most $\frac 32qp$.

Now assume for each $L$ and each $v \in L$ we have revealed whether $v$ is $L$-isolated, but we have not revealed $\chi_L(v)$ when $v$ is non-$L$-isolated. When we do reveal values $\chi_L(v)$ to form the graph $\mathbb{G}_2$, we have for each non-$L$-isolated vertex $v$ that
$\deg_{\mathbb{G}_2[L]}(v) \sim \bin(|\chi_L^{-1}([s])|-1, \frac{s-1}{s})$. Thus,
$\E(\deg_{\mathbb{G}_2[L]}(v)) = (|\chi_L^{-1}([s])|-1)(s-1)/s \ge (\frac 14 p\lambda-1)(s-1)/s$  and the Chernoff bound with $\varepsilon = 1/2$ tells us that 
\begin{align*}
\Pr\of{|\deg_{\mathbb{G}_2[L]}(v) - \E(\deg_{\mathbb{G}_2[L]}(v))| \ge \frac 12 \E(\deg_{\mathbb{G}_2[L]}(v))} &\le 2\exp\left\{- \frac{\E(\deg_{\mathbb{G}_2[L]}(v))}{12} \right\}\\ 
&\le 2\exp\left\{- \frac{(\frac 14 p\lambda-1)(s-1)/s}{12} \right\}
\end{align*}
so by the union bound, the probability that there is a vertex $v$ with $|\deg_{\mathbb{G}_2[L]}(v) - \E(\deg_{\mathbb{G}_2[L]}(v))| \ge \frac 12 \E(\deg_{\mathbb{G}_2[L]}(v))$ is at most 
\[
q^2 \cdot 2\exp\left\{- \frac{(\frac 14 p\lambda-1)(s-1)/s}{12} \right\} = o(1),
\]
since $p \gg (\log q) / \lambda$. Thus \whp for every vertex $v$ and line $L$ for which $v$ is non-$L$-isolated we have that $\deg_{\mathbb{G}_2[L]}(v)=\Theta(p \lambda)$. 

Thus, for each vertex $v$, \whp its degree in $\mathbb{G}_2$ is $\Theta( \lambda q p^2)$.

\medskip

\textbf{\ref{G2:c}:} Recall that $a$ is some positive integer and $b=\left \lceil(s+r)/ { \binom{a-1}{2} }\right \rceil +r$. 
First we show that every $K_{s+r}$ in $\mathbb{G}_2$ contains a dangerous subset in $\mathcal{D}_a \cup \mathcal{D}_b$.
Let $K$ be any copy of $K_{s+r}$ in $\mathbb{G}_2$. Similarly to the graph $\mathbb{G}_1$, we can conclude that the vertices of $K$ must contain a dangerous set in $\mathcal{D}_a$ or in $\mathcal{D}_b$.

In order for such a dangerous set to give a $K_{s+r}$ in $\mathbb{G}_2$  none of the vertices can be $L$-isolated in any of the lines in the dangerous set. For a Type 1 dangerous set, there are $\binom{a}{2}$ lines, each containing $2$ vertices, so the probability that no vertex $v$ is $L$-isolated for $L$ containing $v$ is $p^{2 \binom{a}{2}} = p^{a^2 - a}$. Thus, the expected number of copies of $K_{s+r}$ that arise from Type 1 dangerous sets  is at most $4\lambda^{a}q^a p^{a^2-a}.$ 
Now for dangerous sets of Type 2, the probability that a dangerous set in $\mathcal{D}_b$ gives rise to a $K_{s+r}$ is $p^x$, where $x$ is the number of point-line incidences there are within the dangerous set. As in Section~\ref{sec:G1} the number of point-line incidences is at least $  (2r+1)b - 4r^3$. Therefore the expected number of copies of $K_{s+r}$ arising from Type~2 dangerous sets is at most $4\lambda^b q^{r+2} p^{(2r+1)b - 4r^3}$.

Thus by linearity of expectation the total number of copies of $K_{s+r}$ has expectation at most 
\[
4\lambda^{a}q^a p^{a^2-a} + 4\lambda^b q^{r+2} p^{(2r+1)b - 4r^3}
\] 
and so we are done by the Markov bound applied with $\zeta=2$.

\end{proof}

\subsection{Deriving the lower bound for intermediate values of  $\delta$}

\begin{theorem}\label{thm:main2}
Let $r\ge 1$, $\varepsilon>0$, and $\frac 12 + \frac{1}{2(2r+1)} < \delta < 1$. Then for all sufficiently large $s$, we have 
\[
\RT_s(n, K_{s+r}, n^\delta) = \Omega\of{n^{2 - (1-\delta)/r  - \varepsilon}}.
\]
\end{theorem}

First we state a proposition that we will use to estimate fractions that have ``error" in the numerator and denominator. 

\begin{proposition} \label{estlem}

For any real numbers $x, y, \epsilon_x, \epsilon_y$, if $x,y \neq 0$ and $|\frac{\epsilon_x}{x}|, |\frac{\epsilon_y}{y}| 
\le \frac{1}{2}$, then 
\[
\left|\frac{x+ \epsilon_x}{y+\epsilon_y} - \frac{x}{y} \right| \le \frac{|\epsilon_x y| + 3|\epsilon_y x|}{y^2}. 
\]
\end {proposition}

\begin{proof}
Observe that 
\begin{align*}
\left|\frac{x+ \epsilon_x}{y+\epsilon_y} - \frac{x}{y}\right| &= \left|\frac{x}{y} \left[ \left(1+\frac{\epsilon_x}{x} \right) \cdot \frac{1}{1+\frac{\epsilon_y}{y}} -1 \right] \right|\\
&= \left|\frac{x}{y} \left[ \left(1+\frac{\epsilon_x}{x} \right) \cdot \left(1 + \sum_{n=1}^{\infty} \of{-\frac{\epsilon_y}{y}}^n \right) -1 \right]\right|\\
&= \left|\frac{x}{y} \left[ \frac{\epsilon_x}{x}  +\of{1+\frac{\epsilon_x}{x}}\sum_{n=1}^{\infty} \of{-\frac{\epsilon_y}{y}}^n \right]\right|\\
&\le\left| \frac{ \epsilon_x }{y}\right| + \left|  \frac{x}{y} \right| \cdot \left|1+\frac{\epsilon_x}{x}\right| \cdot \sum_{n=1}^{\infty} \left|\frac{\epsilon_y}{y}\right|^n\\
&\le\left| \frac{ \epsilon_x }{y}\right| + \left|  \frac{x}{y} \right| \cdot \frac{3}{2} \cdot 2 \left|\frac{\epsilon_y}{y}\right|.
\end{align*}
In the last line we have used $|\epsilon_x / x| \le 1/2$ and
\[
\sum_{n=1}^{\infty} \left|\frac{\epsilon_y}{y}\right|^n = \left|\frac{\epsilon_y}{y}\right| \sum_{n=0}^{\infty} \left|\frac{\epsilon_y}{y}\right|^n \le \left|\frac{\epsilon_y}{y}\right| \cdot 2
\] 
which follows from $|\epsilon_y / y| \le 1/2$.
\end{proof}

\begin{proof}[Proof of Theorem~\ref{thm:main2}]
We will apply Lemma~\ref{lem:5} after discussing how to set parameters.

Fix $r\ge 1$, $\varepsilon>0$ and $\frac 12 + \frac{1}{2(2r+1)} < \delta < 1$. Set 
\[
a = 2 + \max\left\{ \left\lceil \frac{1}{\delta}  \right\rceil , \;\;\left\lceil  \of{\frac{(1-\delta)(2r+1)}{\delta(2r+1)-1} + 1} \cdot \frac{\delta(2r+1)-1}{1-\delta} \right\rceil \right\}
\]
and
\[
b=\left \lceil(s+r) / \binom{a-1}{2} \right \rceil +r.
\] 
We will assume that $s$, and consequently $b$, is large enough which satisfies for instance the following
\begin{equation}\label{thm:main2:eps}
\varepsilon > 104r^2/b.
\end{equation}

Similarly as in the proof of Theorem~\ref{thm:main1} for large $n$ there is a prime number $q$ such that 
\[
n^{1/\of{\frac{(1-\delta)(2r+1)b - 4(1-\delta)r^3 +r+3}{\left(\delta(2r+1)-1\right)b - 4\delta r^3}+1}} \le q
\le (1+o(1)) n^{1/\of{\frac{(1-\delta)(2r+1)b - 4(1-\delta)r^3 +r+3}{\left(\delta(2r+1)-1\right)b - 4\delta r^3}+1}}.
\]
(It will be shown soon that the exponents are positive.)
Let 
\[
\lambda = q^\frac{(1-\delta)(2r+1)b - 4(1-\delta)r^3 +r+3}{\left(\delta(2r+1)-1\right)b - 4\delta r^3} \quad \text{ and }\quad 
p = \kappa q^{1 - \delta \of{\frac{(1-\delta)(2r+1)b - 4(1-\delta)r^3 +r+3}{\left(\delta(2r+1)-1\right)b - 4\delta r^3}+1}},
\]
where $\kappa = 20 s \log s$. Finally set $\alpha = n^\delta $. 

We will show that the above parameters satisfy all assumptions of Lemma~\ref{lem:5} implying the existence of a graph $G_2$ of order $\lambda q (1+o(1))$ satisfying \ref{G2:a}-\ref{G2:c}.

First observe that 
\[
\frac{(10s\log s) q}{p} = \frac{q^{\delta\of{\frac{(1-\delta)(2r+1)b - 4(1-\delta)r^3 +r+3}{\left(\delta(2r+1)-1\right)b - 4\delta r^3}+1}}}{2}
=\frac{(1+o(1))n^{\delta}}{2} \le \alpha.
\]
yielding by~\ref{G2:a} that the $s$-independence number of $G_2$ is less than $\alpha$. 

Now we examine the exponent of~$\lambda$. Clearly, 
\[
\lim_{b\to\infty} \frac{(1-\delta)(2r+1)b - 4(1-\delta)r^3 +r+3}{\left(\delta(2r+1)-1\right)b - 4\delta r^3} = \frac{(1-\delta)(2r+1) }{\delta(2r+1)-1}.
\]
We will show that for $b \ge 20r^2$ the exponent is very close to the above limit. Indeed, since $r \ge 1$ and $1/2  < \delta < 1$, we have 
\[
\left|\frac{  -4(1-\delta)r^3 +r+3}{(1-\delta)(2r+1)b} \right| \le \frac{8r^3}{rb} \le \frac 12 \quad \text{ and }  \quad  \left|\frac{-4\delta r^3}{(\delta(2r+1)-1)b} \right| \le \frac{8r^3}{rb} \le \frac 12
\]
and so by Proposition~\ref{estlem} we obtain
\begin{align*}
&\left|\frac{(1-\delta)(2r+1)b - 4(1-\delta)r^3 +r+3}{(\delta(2r+1)-1)b - 4\delta r^3} - \frac{(1-\delta)(2r+1) }{\delta(2r+1)-1} \right| \\ 
&\qquad\qquad\qquad\le \frac{\of{4(1-\delta)r^3 +r+3}(\delta(2r+1)-1)b + 3\cdot 4\delta r^3(1-\delta)(2r+1)b  }{\of{(\delta(2r+1)-1)b}^2}\\
&\qquad\qquad\qquad\le \frac{\of{8r^3}(2r)b + 12r^3(3r)b  }{\of{rb}^2} = \frac{52r^2}{b},
\end{align*}
where the last inequality follows from $r \ge 1$ and $1/2 < \delta < 1$.  Also note that  
\[
\frac{(1-\delta)(2r+1) }{\delta(2r+1)-1} \ge 1-\delta
\] 
and
\[
\frac{(1-\delta)(2r+1) }{\delta(2r+1)-1} = 1- \frac{(2\delta-1)(2r+1)-1}{\delta(2r+1)-1} \le  1- \frac{(2\delta-1)(2r+1)-1}{2r}.
\]
Observe that since $\delta > \frac 12 + \frac{1}{2(2r+1)}$, we get $\frac{(2\delta-1)(2r+1)-1}{2r} > 0$. Thus, if 
\[
\frac{52r^2}{b} < \min\left\{1-\delta, \frac{(2\delta-1)(2r+1)-1}{2r} \right\},
\] 
then
\begin{equation}\label{eq:main2:1}
\frac{(1-\delta)(2r+1)b - 4(1-\delta)r^3 +r+3}{(\delta(2r+1)-1)b - 4\delta r^3} \le \frac{(1-\delta)(2r+1) }{\delta(2r+1)-1} + \frac{52r^2}{b} < 1
\end{equation}
and
\begin{equation}\label{eq:main2:2}
\frac{(1-\delta)(2r+1)b - 4(1-\delta)r^3 +r+3}{(\delta(2r+1)-1)b - 4\delta r^3} \ge \frac{(1-\delta)(2r+1) }{\delta(2r+1)-1} - \frac{52r^2}{b} > 0.
\end{equation}
Consequently, $\lambda$ is at least $q^{\Omega(1)}$ and less than $q$, so $\log q \ll \lambda \le q$, as required in Lemma~\ref{lem:5}.

Furthermore, 
\[
\lambda p = \kappa q^{1+(1-\delta)\of{ \frac{(1-\delta)(2r+1)b - 4(1-\delta)r^3 +r+3}{\left(\delta(2r+1)-1\right)b - 4\delta r^3} } - \delta}
\ge \kappa q^{1-\delta} \gg \log q
\]
and
\begin{align*}
p &= \kappa q^{1 - \delta \of{\frac{(1-\delta)(2r+1)b - 4(1-\delta)r^3 +r+3}{\left(\delta(2r+1)-1\right)b - 4\delta r^3}+1}}
\overset{\eqref{eq:main2:2}}{\le} \kappa q^{ 1-\delta\of{ \frac{(1-\delta)(2r+1) }{\delta(2r+1)-1} - \frac{52r^2}{b} +1  }  } \\
&= \kappa q^{ 1-\delta\of{ \frac{1-\delta}{\delta} + \frac{1-\delta}{\delta\of{\delta(2r+1)-1}} - \frac{52r^2}{b} +1  }  } 
= \kappa q^{-\frac{1-\delta}{\delta(2r+1)-1} + \frac{52r^2\delta}{b}} = o(1)
\end{align*}
for $b$ sufficiently large. Consequently, $1 \ge p \gg (\log q) / \lambda$ and all assumptions of Lemma~\ref{lem:5} are satisfied.

%Now set $\lambda$ to satisfy the equation 
%\[
%\lambda^b q^{r+2} p^{(2r+1)b-4r^3} = \frac{1}{q}.
%\] 
%More specifically, since $n=\lambda q (1+o(1))$ and $p= \kappa q / n^\delta$, we have 
%\[
%1 = \lambda^b q^{r+3} p^{(2r+1)b-4r^3} = \lambda^b q^{r+3} \of{\frac{\kappa q}{\lambda^\delta q^\delta (1+o(1))}}^{(2r+1)b-4r^3}.
%\]
%Thus,
%\[
%\lambda = \Theta\of{q^\frac{(1-\delta)(2r+1)b - 4(1-\delta)r^3 +r+3}{\left(\delta(2r+1)-1\right)b - 4r^3}}.
%\]

Now we will see that our choice of parameters makes $G_2$ a $K_{s+r}$-free graph. By~\ref{G2:c}, the number of copies of $K_{s+r}$ is at most 
\begin{equation} \label{nums+r2}
8\left(\lambda^a q^a p^{a^2 -a} +  \lambda^b q^{r+2} p^{(2r+1)b - 4r^3}\right).
\end{equation}
%Recall that by our choice of $\lambda$, the second term is just $8 / q = o(1)$. 
The first term is on the order of $\lambda^a q^a p^{a^2 -a} = O\of{\left(\lambda q p^{a-1}\right)^a}$.
We show that $\lambda q p^{a-1} = o(1)$. The order of magnitude of the latter is
\begin{align*}
q^\frac{(1-\delta)(2r+1)b - 4(1-\delta)r^3 +r+3}{\left(\delta(2r+1)-1\right)b - 4\delta r^3} &\cdot q \cdot q^{\of{1 - \delta \of{\frac{(1-\delta)(2r+1)b - 4(1-\delta)r^3 +r+3}{\left(\delta(2r+1)-1\right)b - 4\delta r^3}+1}}(a-1)}\\
&=q^{\of{\frac{(1-\delta)(2r+1)b - 4(1-\delta)r^3 +r+3}{\left(\delta(2r+1)-1\right)b - 4\delta r^3} }\cdot (1-\delta(a-1)) + 1+(1-\delta)(a-1)}  \\
&\overset{\eqref{eq:main2:1}}{\le} q^{\of{ \frac{(1-\delta)(2r+1)}{\delta(2r+1)-1} + \frac{52r^2}{b} }\cdot (1-\delta(a-1)) + 1+(1-\delta)(a-1)}  \\
&=q^{ \frac{(1-\delta)(2r+1)}{\delta(2r+1)-1} + \frac{52r^2}{b}(1- \delta(a-1))  + 1 + \of{1 - \delta - \frac{(1-\delta)(2r+1)}{\delta(2r+1)-1} \delta }(a-1)} \\
&\le q^{ \frac{(1-\delta)(2r+1)}{\delta(2r+1)-1}   + 1 + \of{1 - \frac{(2r+1)\delta}{\delta(2r+1)-1} }(1-\delta)(a-1)} \\
&= q^{ \frac{(1-\delta)(2r+1)}{\delta(2r+1)-1} + 1 -  \frac{1}{\delta(2r+1)-1} (1-\delta)(a-1)} = o(1),
\end{align*}
where the second to last line follows from $a > 1+ 1/\delta$ and the last line follows from
\[
a > 1 + \of{\frac{(1-\delta)(2r+1)}{\delta(2r+1)-1} + 1} \cdot \frac{\delta(2r+1)-1}{1-\delta}.
\] 
Thus, $\lambda^a q^a p^{a^2 -a}=o(1)$.
Now we bound the order of the magnitude of the second term in~\eqref{nums+r2}. Observe that 
\begin{align*}
q^{\frac{(1-\delta)(2r+1)b - 4(1-\delta)r^3 +r+3}{\left(\delta(2r+1)-1\right)b - 4\delta r^3} b} &\cdot q^{r+2} \cdot q^{\of{1 - \delta \of{\frac{(1-\delta)(2r+1)b - 4(1-\delta)r^3 +r+3}{\left(\delta(2r+1)-1\right)b - 4\delta r^3}+1}}\of{(2r+1)b - 4r^3}} = \frac{1}{q} = o(1).
\end{align*}
Thus, $G_2$ is $K_{s+r}$-free.

Now let $G$ be any induced subgraph of $G_2$ of order $n$. Clearly, $G$ is $K_{s+r}$-free with $\alpha_s(G) < n^{\delta}$. Furthermore, since $n = (1+o(1))\lambda q$,
\begin{align*}
|E(G)| &\ge |E(G_2)| - |V(G_2) - V(G)| \cdot \Delta(G_2)\\ 
& \ge |V(G_2)|\cdot \delta(G_2) - o(1) \cdot \Delta(G_2)  = \Omega\of{ \lambda q \cdot \lambda q p^2} = \Omega(np^2),
\end{align*}
by~\ref{G2:b}. Since 
\[
p = (1+o(1))\kappa n^{1/\of{\frac{(1-\delta)(2r+1)b - 4(1-\delta)r^3 +r+3}{\left(\delta(2r+1)-1\right)b - 4\delta r^3}+1}-\delta}
\overset{\eqref{eq:main2:1}}{\ge} (1+o(1))\kappa n^{1/\of{\frac{(1-\delta)(2r+1) }{\delta(2r+1)-1} + \frac{52r^2}{b}+1}-\delta},
\]
we get
\[
\Omega\of{n^2 p^2} 
= \Omega \of{n^{2-2\delta + 2/\of{1+ \frac{(1-\delta)(2r+1) }{\delta(2r+1)-1} + \frac{52r^2}{b}}} }
= \Omega \of{n^{2-2\delta + 2/\of{\frac{2r }{\delta(2r+1)-1} + \frac{52r^2}{b}}} }.
\]
Since for any positive real numbers $x, y, z$ with $y > z$,
\[
\frac{x}{y+z} = \frac{x}{y} \cdot \frac{1}{1+\frac{z}{y}} \ge \frac xy \cdot \of{1- \frac zy} = \frac{x}{y} - \frac{xz}{y^2},
\]
we get
\[
2\left/\right.\!\of{\frac{2r }{\delta(2r+1)-1} + \frac{52r^2}{b}} \ge \frac{\delta(2r+1)-1}{r} - \frac{104r^2}{b} \cdot \of{\frac{\delta(2r+1)-1}{2r}}^2
\ge \frac{\delta(2r+1)-1}{r} - \varepsilon,
\]
the latter is due to~\eqref{thm:main2:eps}.
Thus,
\[
\Omega\of{n^2 p^2} = \Omega\of{ n^{2-2\delta + \frac{\delta(2r+1)-1}{r} - \varepsilon} } = \Omega\of{n^{2-\frac{1-\delta}{r}-\varepsilon}}
\]
completing the proof.
\end{proof}

%%%%%%%%%%%%%%%%%%%%%%%%%%%%%%%%%%%%%%%%%%%%%%%%%%%%%%%%%%%%
%%%%%%%%%%%%%%%%%%%%  Proof of Theorem  %%%%%%%%%%%%%%%%%%%%
%%%%%%%%%%%%%%%%%%%%%%%%%%%%%%%%%%%%%%%%%%%%%%%%%%%%%%%%%%%%

\section{Upper bound on $\RT_s(n, K_{s+r}, n^\delta)$}\label{sec:ub}

First we state the well-known dependent random choice lemma from a survey paper of Fox and Sudakov~\cite{FS}. Early versions of this lemma were proved and applied by various researchers,
starting with Gowers~\cite{Gowers}, R\"odl and Kostochka~\cite{KRodl}, and Sudakov~\cite{SU3,SU,SU2}.

\begin{lemma}\label{lem:dependent}
Let $a, d, m, n, r$ be positive integers. Let $G = (V,E)$ be a graph with $|V| = n$ vertices and
average degree $d = 2|E(G)|/n$. If there is a positive integer $t$ such that
\[
\frac{d^t}{n^{t-1}} - n^r \left( \frac{m}{n} \right)^{t} \ge a,
\]
then $G$ contains a subset $U$ of at least $a$ vertices such that every $r$ vertices in $U$ have at least $m$
common neighbors.
\end{lemma}

\begin{corollary}\label{cor:sudakov}
Let $r\ge 1$ be an integer and $0<\delta<1$.
Let $G$ be a graph on $n$ vertices with at least $n^{2-(1-\delta)/\lceil \frac{r-\delta}{1-\delta} \rceil}$ edges. 
Then $G$ contains a subset $U$ of at least $n^{\delta}$ vertices such that every $r$ vertices in $U$ have at least $n^{\delta}$
common neighbors.
\end{corollary}

\begin{proof}
Let $t=\lceil \frac{r-\delta}{1-\delta} \rceil$, $a=m=n^{\delta}$, and $d=2n^{1-(1-\delta)/\lceil \frac{r-\delta}{1-\delta} \rceil}$. Then,
\[
\frac{d^t}{n^{t-1}} - {n^r} \left( \frac{m}{n} \right)^{t} 
= 2^{\lceil \frac{r-\delta}{1-\delta} \rceil} n^{\delta}  - n^r n^{-(1-\delta)\cdot \lceil \frac{r-\delta}{1-\delta} \rceil}
\ge 2^{\lceil \frac{r-\delta}{1-\delta} \rceil} n^{\delta} - n^{\delta}
\ge a.
\]
Now the corollary follows from the above lemma.
\end{proof}

The next theorem is an easy generalization of a result of Sudakov from~\cite{SU}.
\begin{theorem}\label{thm:ub_delta}
Let $s \ge r\ge 1$ and $0<\delta<1$. Then,
\[
\RT_s(n,K_{s+r}, n^{\delta}) < 
\begin{cases}
n^{2-\frac{(1-\delta)^2}{r-\delta}}, & \text{ if $\frac{r-\delta}{1-\delta}$ is an integer},\\ 
n^{2-\frac{(1-\delta)^2}{r+1-2\delta}}, & \text{ otherwise}. 
\end{cases}
\]
\end{theorem}

\begin{proof}
Clearly if $\frac{r-\delta}{1-\delta}$ is an integer, then $n^{2-\frac{(1-\delta)^2}{r-\delta}} = n^{2-(1-\delta)/\lceil \frac{r-\delta}{1-\delta} \rceil}$. Otherwise, 
\[
n^{2-\frac{(1-\delta)^2}{r+1-2\delta}} = n^{2-(1-\delta)/ \frac{r+1-2\delta}{1-\delta} }
= n^{2-(1-\delta)/ \left(\frac{r-\delta}{1-\delta} +1\right)}
> n^{2-(1-\delta)/\lceil \frac{r-\delta}{1-\delta} \rceil}.
\]

Let $G$ be a graph on $n$ vertices with at least $n^{2-(1-\delta)/\lceil \frac{r-\delta}{1-\delta} \rceil}$ edges which contains no copy of $K_{s+r}$. We show that $\alpha_s(G) \ge n^{\delta}$. By Corollary~\ref{cor:sudakov} graph $G$ contains a subset of vertices $U$ of size~$n^\delta$ such that any $W\subseteq U$ of size $r$ has $|N(W)|\ge n^{\delta}$. If $G[U]$ contains no $K_s$, then $U$ is an $s$-independent set and we are done. So suppose it contains a copy of $K_s$ and denote by $W$ any $r$ vertices of such copy (recall that $r\le s$). Clearly, $|N(W)|\ge n^{\delta}$. If $N(W)$ contains a copy of~$K_s$, then together with
the vertices in $W$ we obtain a complete subgraph of $G$ on $s+r$ vertices, a contradiction. Thus $N(W)$ is an $s$-independent set of size at least $n^{\delta}$.
\end{proof}

%%%%%%%%%%%%%%%%%%%%%%%%%%%%%%%%%%%%%%%%%%%%%%%%%%%%%%%%%%%%
%%%%%%%%%%%%%%%%%%%%  Proof of Theorem  %%%%%%%%%%%%%%%%%%%%
%%%%%%%%%%%%%%%%%%%%%%%%%%%%%%%%%%%%%%%%%%%%%%%%%%%%%%%%%%%%

\section{Better lower bounds on $\RT_s(n,K_{s+1}, n^{\delta})$ for certain values of $\delta$}\label{sec:r_eq_1}

Observe that Theorems \ref{thm:main1}, \ref{thm:main2} and \ref{thm:ub_delta} (applied with $r=1$) immediately yield the following statement.
\begin{theorem}\label{thm:r_eq_1}
Let $\varepsilon>0$ and $\frac 12 < \delta < 1$. Then for all sufficiently large $s$, we have 
\[
\Omega\of{n^{1+\delta  - \varepsilon}} = \RT_s(n, K_{s+1}, n^\delta) = O\of{n^{1+\delta}}.
\]
\end{theorem}
As it was already observed this is also optimal with respect to $\delta$, since for $\delta\le 1/2$ $\RT_s(n, K_{s+1}, n^\delta)=0$. We will show now that for specific values of $\delta$ one can basically remove $\varepsilon$ from the exponent. 

First we recall some basic properties of generalized quadrangles. A \textit{generalized quadrangle} of order $(p,q)$ is an incidence structure on a set~${P}$ of points and a set~$\mathcal{L}$ of lines such that:
\begin{enumerate}[label=$({\textup{Q}}\textup{\arabic*})$]
\item any two points lie in at most one line,
\item\label{Q:2} if $u$ is a point not on a line~$ L$, then there is a unique point $w\in L$ collinear with~$u$, and hence, no three lines form a triangle,
\item every line contains~$p+1$ points, and every point lies on~$q+1$ lines,
\item\label{Q:4} $|P| = (pq+1)(p+1)$ and $|\mathcal{L}| = (pq+1)(q+1)$,
\item\label{Q:5} $p\le q^2$ and $q\le p^2$.
\end{enumerate}

\begin{theorem}\label{thm:quad}
Let $s\ge 2$. If a generalized quadrangle of order $(p,q)$ exists, then
\begin{equation}\label{thm:quad:1}
\RT_s((pq+1)(p+1), K_{s+1}, \Theta(pq)) = \Theta(p^3q^2),
\end{equation}
where the hidden constants depend only on $s$.
\end{theorem}

\begin{proof}
For a given generalized quadrangle $(P,\mathcal{L})$ we construct the random graph $\mathbb{G}=(V,E)$ with $V=P$ as follows. For every $L\in \mathcal{L}$, let $\chi_L : L \to  {[s]}$ be a random partition of the vertices of $L$ into $s$ classes chosen uniformly at random, where the classes need not have the same size and the unlikely event that a class is empty is permitted. Next we embed the $s$-partite complete graph on $\chi^{-1}_L (1) \cup \chi^{-1}_L (2) \cup \dots \cup \chi^{-1}_L (s)$. Observe that not only are $\mathbb{G}[L]$ and $\mathbb{G}[L']$ edge disjoint for distinct $L, L' \in \mathcal{L}$, but also that the partitions for $L$ and $L'$ were determined independently.

Observe that by~\ref{Q:2} $\mathbb{G}$ is $K_{s+1}$-free and by the Chernoff bound \whp $|E| = |\mathcal{L} |\cdot \Omega(p^2) = \Omega(p^3q^2)$.

Now we will show that $\alpha_s(\mathbb{G}) \le s^2 p q$. Consider any $C \in {V \choose s^2 pq}$. We will bound the probability that $\mathbb{G}[C] \not \supseteq K_s$. 
For each $L \in \mathcal{L}$, let $X_L$ be the event that $K_s \not \subseteq \mathbb{G}[L \cap C]$. Clearly, 
\[
\Pr(X_L) \le s \left( 1-\frac{1}{s}\right)^{|L\cap C|} \le s e^{-|L\cap C|/s}
\]
and by independence, 
\[
\Pr \Big( K_s \not \in \mathbb{G}[C] \Big) 
\leq \Pr \Big( \bigcap_{L \in \mathcal{L}} X_L \Big) 
\leq \prod_{L\in \mathcal{L}} s e^{-|L\cap C|/s}
\leq s^{|\mathcal{L}|} e^{-\sum_{L\in \mathcal{L}} |L\cap C|/s }.
%\leq s^{|\mathcal{L}|}{\left( 1- \frac{1}{s} \right) }^{|\mathcal{L}'_{C,s}|} \leq {\left( 1- \frac{s!}{s^s} \right) }^\frac{\lambda q}{8}\leq \exp { \left\{ - \frac{s!}{s^s}\frac{\lambda q }{8} \right\} }.
\]
Since $\sum_{L\in \mathcal{L}} |L\cap C| = |C|(q+1)$, we get
\[
\Pr \Big( K_s \not \in \mathbb{G}[C] \Big) 
\leq s^{|\mathcal{L}|} e^{-|C|(q+1)/s}
= e^{|\mathcal{L}| \log s -|C|(q+1)/s}.
\]
So by the union bound, the probability that there exists a subset of $s^2 p q$ vertices in $V$ that contains no $K_s$ is at most
\[
\binom{|V|}{s^2 pq} e^{|\mathcal{L}| \log s -|C|(q+1)/s}
\le  |V|^{s^2 pq} e^{|\mathcal{L}| \log s -|C|(q+1)/s}
= e^{s^2 pq \log |V| + |\mathcal{L}| \log s -|C|(q+1)/s} = o(1),
\]
because of~\ref{Q:4} and $|C| = s^2pq$. Thus, \whp $\alpha_s(\mathbb{G}) \le s^2 pq$ and consequently
\[
\RT_s((pq+1)(p+1), K_{s+1}, \Theta(pq)) = \Omega(p^3q^2).
\]

Finally observe that the upper bound in~\eqref{thm:quad:1} is trivial, since if $G$ is a $K_{s+1}$-free graph of order $(pq+1)(p+1)$, then $\Delta(G) < \alpha_s(G) = O(pq)$ yielding $|E(G)| = O(p^3q^2)$.
\end{proof}

It is known that when $(p,q)\in \{(r,r), (r,r^2),(r^2,r),(r^2,r^3),(r^3,r^2)\} $ for any arbitrary prime power $r$, then the generalized quadrangle exists (see, e.g., ~\cite{GO,TH}) yielding the following:

\begin{center}
\begin{tabular}{c | c | c | c | c }
$p$ & $q$ & $(pq+1)(p+1)$  & $pq$ & $p^3 q^2$ \\
\hline\hline
$r$ & $r$ & $ \Theta(r^3)$ & $r^2$ & $r^5$\\ \hline
$r$ & $r^2$ & $ \Theta(r^4)$ & $r^3$ & $r^7$\\ \hline
$r^2$ & $r$ & $ \Theta(r^5)$ & $r^3$ & $r^8$\\ \hline
$r^2$ & $r^3$ & $ \Theta(r^7)$ & $r^5$ & $r^{12}$\\ \hline
$r^3$ & $r^2$ & $ \Theta(r^8)$ & $r^5$ & $r^{13}$\\ \hline
\end{tabular}
\end{center}

\smallskip

Thus, by letting $n$ to be $\Theta(r^3)$, $\Theta(r^4)$, $\Theta(r^5)$, $\Theta(r^7)$, and $\Theta(r^8)$, respectively,  we get the following corollary.
\begin{corollary}\label{cor:quad}
Let $\delta \in \{3/5, 5/8, 2/3, 5/7, 3/4\}$. Then
\[
\RT_s(n,K_{s+1}, \Theta(n^{\delta})) = \Theta(n^{1+\delta}).
\]
\end{corollary}

%In general, assume that $(pq+1)(p+1) = n$ and $p = n^{1-\delta}$. Then, $q=\Theta(n^{2\delta-1})$ and so $\Theta(pq) = \Theta(n^{\delta})$ and $\Theta(p^3q^2) = \Theta(n^{1+\delta})$. Then, \ref{Q:5} implies that
%\[
%n^{1-\delta} = p\le q^2 = \Theta(n^{4\delta-2}) \quad \text{ and } \quad \Theta(n^{2\delta-1}) = q\le p^2 = n^{2-2\delta}.
%\]
%Thus,
%\[
%1-\delta \le 4\delta - 2 \quad \text{ and } \quad 2\delta-1 \le 2-2\delta
%\]
%and hence $3/5 \le \delta \le 3/4$. Hence, theoretically the above construction could allow to show that $\RT_s(n,K_{s+1}, \Theta(n^{\delta})) = \Theta(n^{1+\delta})$ for $3/5 \le \delta \le 3/4$.

In view of Theorems~\ref{thm:r_eq_1} and \ref{thm:quad} and the above corollary we conjecture that actually
\[
\RT_s(n,K_{s+1}, n^{\delta}) = \Theta(n^{1+\delta})
\] 
for any $1/2 < \delta < 1$.

%%%%%%%%%%%%%%%%%%%%%%%%%%%%%%%%%%%%%%%%%%%%%%%%%%%%%%%%%%%%
%%%%%%%%%%%%%%%%%%%%  Proof of Theorem  %%%%%%%%%%%%%%%%%%%%
%%%%%%%%%%%%%%%%%%%%%%%%%%%%%%%%%%%%%%%%%%%%%%%%%%%%%%%%%%%%

\section{Phase transition of $\RT_s(n, K_{2s+1}, f)$}\label{sec:big_r}

Throughout this section $\omega = \omega(n)$ is a function which goes to infinity arbitrarily slowly together with $n$.

Here we briefly discuss an extension  of a recent result of Balogh, Hu, and Simonovits~\cite{BHS}, who showed (answering a question of Erd\H{o}s and S\'os~\cite{ES}) that
\begin{equation}\label{eq:BHS:1}
\RT_2(n,K_5, \sqrt{n\log n} / \omega) = o(n^2)
\end{equation}
and
\begin{equation}\label{eq:BHS:2}
\RT_2(n,K_5, c \sqrt{n\log n}) = n^2/4 + o(n^2)
\end{equation}
for any $c>1$. Thus, one can say that $K_5$ has a (Ramsey-Tur\'an) phase transition at $c\sqrt{n\log n}$ for every $c>1$. One can easily show that for any $s\ge 2$, $K_{2s+1}$ also has a phase transition (with respect to the $s$-independence number). This will follow from the following two theorems.

\begin{theorem}\label{thm:ub_ks_r}
Let $s\ge 2$. Then 
\[
\RT_s(n,K_{2s+1}, \f_{s,s+1}(n/\omega)) = o(n^2).
\]
\end{theorem}

First we derive another corollary from Lemma~\ref{lem:dependent}.

\begin{corollary}\label{cor:dependent2}
Let $\varepsilon>0$ be fixed and $s\ge 1$ be an integer. Let $G$ be a graph on $n$ vertices with at least $\varepsilon n^2$ edges. 
Then for sufficiently large $n$, $G$ contains a subset $U$ of at least $n/\omega$ vertices such that every $s+1$ vertices in $U$ have at least $\f_{s,s+1}(n/\omega)$ common neighbors.
\end{corollary}

\begin{proof}
Let $r=s+1$, $t=2s+1$, $a = n/\omega$, $m = \f_{s,s+1}(n/\omega)$, and $d=2\varepsilon n$. Observe that by~\eqref{f:1}, $\f_{s,s+1}(n/\omega) = o(n^{(s+1)/(2s+1)})$.
Thus,
\[
\frac{d^t}{n^{t-1}} - {n^r} \left( \frac{m}{n} \right)^{t} 
\ge (2\varepsilon)^{2s+1} n - \frac{(\f_{s,s+1}(n/\omega))^{2s+1}}{n^s} \ge n/\omega = a,
\]
for sufficiently large~$n$.
Now the corollary follows from Lemma~\ref{lem:dependent}.
\end{proof}

\begin{proof}[Proof of Theorem~\ref{thm:ub_ks_r}]
Let $\varepsilon>0$ be an arbitrarily small constant and let $G$ be a $K_{2s+1}$-free graph of order $n$ with $\varepsilon n^2$ edges and $\alpha_s(G)< \f_{s,s+1}(n/\omega)$. By Corollary~\ref{cor:dependent2} there is a set $U$ of size $n/\omega$ such that every $s+1$ vertices in $U$ have at least $\f_{s,s+1}(n/\omega)$ common neighbors. First observe that $G[U]$ contains a copy of $K_{s+1}$. Otherwise, $G[U]$ is a $K_{s+1}$-free graph of order $n/\omega$ with $\alpha_s(G[U])<\f_{s,s+1}(n/\omega)$. But clearly this cannot happen.
Now let $W$ be the set of vertices of a copy of $K_{s+1}$ in $U$. If $N(W)$ contains a copy of $K_s$, then such a copy together with $G[W]$ gives a copy of $K_{2s+1}$. Thus, $G[N(W)]$ is $K_s$-free. But this yields that $\alpha_s(G)\ge |N(W)|\ge \f_{s,s+1}(n/\omega)$, a contradiction.
\end{proof}

The following extends a result of Erd\H{o}s, Hajnal, Simonovits, S\'os, and Szemer\'edi~\cite{EHSSS} for small $s$-independence numbers.

\begin{theorem}\label{thm:lb_ks_r}
Let $k\ge 2$, $s\ge 2$, $1\le r\le k-1$. Let $c$ be a constant satisfying
\[
c > 
\begin{cases}
1 \text{ if } s=2,\\
k \text{ if } s\ge 3.\\
\end{cases}
\]
Then,
\[
\RT_s(n,K_{ks+r}, c\;\!\f_{s,s+1}(n/k)) \ge \frac{1}{2} \left(1-\frac{1}{k} \right) n^2 + o(n^2). 
\]
Furthermore, if $r=1$, then 
\[
\RT_s(n,K_{ks+1}, c\;\!\f_{s,s+1}(n/k)) = \frac{1}{2} \left(1-\frac{1}{k} \right) n^2 + o(n^2).
\]
\end{theorem}

\begin{proof}
First we assume that $s\ge 3$ and $c > k$.
Let $H$ be a $K_{s+1}$-free graph of order $n/k  + o(n)$ with $\alpha_s(H) < c/k\cdot\f_{s,s+1}(n/k)$. The existence of such graph follows from the definition of the Erd\H{o}s-Rogers number. Let $G=(V,E)$ be a graph of order $n$ such that $V=V_1\cup\dots\cup V_{k}$, $|V_1| = \dots = |V_k| = n/k+o(n)$, $G[V_i]$ is isomorphic to $H$ for each $1\le i\le k$, and for each $v\in V_i$ and $w\in V_j$ we have $\{v,w\}\in E$ for any $1\le i < j\le k$. Observe that $G$ is $K_{ks+r}$-free and $\alpha_{s}(G) < c\;\!\f_{s,s+1}(n/k)$. Thus,
\[
\RT_s(n,K_{ks+r}, c\;\!\f_{s,s+1}(n/k)) \ge \binom{k}{2} \left( \frac{n}{k} \right)^2 + o(n^2) = \frac{1}{2} \left(1-\frac{1}{k} \right) n^2 + o(n^2).
\]

Now observe that if $s=2$ and $c>1$, then it suffices to assume that $\alpha_s(H) < c\;\!\f_{s,s+1}(n/k)$, since any independent set in $G$ must lie entirely in some $V_i$.

Finally, if $r=1$, then
\[
\RT_s(n,K_{ks+1}, c\;\!\f_{s,s+1}(n/k)) \le \RT_s(n,K_{ks+1}, o(n)) \le \frac{1}{2} \left(1-\frac{1}{k} \right) n^2 + o(n^2),
\]
where the latter follows from Theorem 2.6\,(a) in~\cite{EHSSS}.
\end{proof}

Since by~\eqref{f:1b} we have  $ \f_{2,3}(n/2) \le (1+o(1))\sqrt{n\log n}$ and $  \f_{2,3}(n/\omega) \ge \sqrt{n\log n}/\omega$, the above theorems can be viewed as generalizations of \eqref{eq:BHS:1} and \eqref{eq:BHS:2}, which are the case $k=s=2$ and $r=1$.

In particular, when $k=2$ and $r=1$, we get that $\RT_s(n,K_{2s+1}, n^{\delta}) = n^2/4 + o(n^2)$  for any $1/2<\delta<1$ and $\RT_s(n,K_{2s+1}, n^{\delta})=o(n^2)$ for $\delta \le 1/2$. Also observe that if we replace $K_{2s+1}$ by $K_{2s}$, then $\RT_s(n,K_{2s}, n^{\delta}) = o(n^2)$ for any $\delta<1$ (by Theorem \ref{thm:ub_delta}).
Here we bound $\RT_s(n,K_{2s}, n^{\delta})$ from below using the Zarankiewicz function.
Recall that the \emph{Zarankiewicz function} $\z(n,s)$ denotes the maximum possible number of edges in a $K_{s,s}$-free bipartite graph $G = (U\cup V, E)$ with $|U| = |V| = n$.

\begin{theorem}\label{thm:2s}
Let $s\ge 2$ and $c>2$. Then
\[
\RT_s(n,K_{2s}, c\;\!\f_{s,s+1}(n/2)) \ge \z(n/2,s).
\]
\end{theorem}

\begin{proof}
Let $H$ be a $K_{s+1}$-free graph of order $n/2$ with $\alpha_s(H) < c/2 \cdot\f_{s,s+1}(n/2)$. Let $G=(V,E)$ be a graph of order $n$ such that $V=V_1\cup V_2$, $|V_1| = |V_2| = n/2$, and $G[V_1]$ and $G[V_2]$ are isomorphic to $H$. Between $V_1$ and $V_2$ we embed a $K_{s,s}$-free bipartite graph with $z(n/2,s)$ edges. Clearly $G$ is $K_{2s}$-free and $\alpha_{s}(G) < c\;\!\f_{s,s+1}(n/2)$.  Thus,
\[
\RT_s(n,K_{2s}, c\;\!\f_{s,s+1}(n/2)) \ge |E| \ge \z(n/2, s).
\]
\end{proof}

It is known due to a result of Bohman and Keevash~\cite{BK} that for sufficiently large $n$ there exists a $K_{s,s}$-free graph of order~$n$ with 
$\Omega\left(n^{2-2/(s+1)} \cdot (\log n)^{1/(s^2-1)}\right)$ edges. Since each graph has a bipartite subgraph with at least half of the edges, we get that 
\[
\z(n, s) \ge \frac{1}{2}\ex(n,K_{s,s}) = \Omega\left(n^{2-2/(s+1)} \cdot (\log n)^{1/(s^2-1)}\right).
\]
This together with Theorem~\ref{thm:2s} implies that 
\begin{equation}\label{eq:K2s}
\RT_s(n,K_{2s}, f) = \Omega\left(n^{2-2/(s+1)} \cdot (\log n)^{1/(s^2-1)}\right)
\end{equation}
for any $s\ge 2$ and $f \ge c\;\!\f_{s,s+1}(n/2)$ with $c>2$.

\section{Concluding remarks}

The study of the Ramsey-Tur\'an number for small $s$-independence number brings several new problems and challenges. We believe that now the most interesting question is to decide whether the lower bounds given by Theorems~\ref{thm:main1} and~\ref{thm:main2} are nearly optimal for any $r\ge 2$. We believe that the upper bound in Theorem~\ref{thm:ub_delta} can be improved.

It is also  not hard to verify that the proofs of Theorems~\ref{thm:main1} and~\ref{thm:main2} can allow $r$ to grow with $s$. For fixed $\varepsilon$  and $\delta$, these theorems still hold for $r = r(s) \le c s^{1/5}$ for some small constant $c$ which may depend on $\varepsilon$ and $\delta$. Of course, since $\RT_s(n, K_{s+r}, n^\delta)$ is nondecreasing in $r$, we get some lower bound for larger $r$ as well, but this fact is not particularly satisfying. While we did not put much effort into allowing $r$ to grow large with $s$, new ideas would be needed to prove lower bounds on $\RT_s(n, K_{s+r}, n^\delta)$ that get better as $\delta$ increases and which apply to say, all $1 \le r \le s$. (Recall that the upper bound in Theorem~\ref{thm:ub_delta} applies to all $1 \le r \le s$, and that we have the lower bound~\eqref{eq:K2s} but this bound stays the same even for $\delta$ quite close to $1$.)

\section{Acknowledgment}
We are grateful to an anonymous referee for bringing to our attention reference~\cite{BHS}.

%%%%%%%%%%%%%
%%%%%%%%%%%%%%%%%%%%%%%%%%%%%%%%%%%%%%%%%%%%%%%
%%%%%%%%%%%%%%%%%%%%  BIBLIOGRAPHY  %%%%%%%%%%%%%%%%%%%%%%%%
%%%%%%%%%%%%%%%%%%%%%%%%%%%%%%%%%%%%%%%%%%%%%%%%%%%%%%%%%%%%

\end{document}